\documentclass[11pt,british,a4paper,reqno]{amsart}
    \usepackage[T1]{fontenc}
\usepackage[utf8]{inputenc}
	\usepackage[british]{babel}
    \DeclareMathSizes{12}{12}{7}{6}

\usepackage{amssymb}
\usepackage{amsmath} 
\usepackage{braket}
\usepackage{mathrsfs}
\usepackage{graphicx}
\usepackage{transparent}
\usepackage{amsthm}
\usepackage{geometry}
\usepackage{xcolor}
\usepackage[shortlabels]{enumitem}
\usepackage{mathtools}
\usepackage{setspace}
\usepackage{tikz}
\usepackage{tikz-cd}

\usepackage[hidelinks]{hyperref}
\usepackage[capitalise]{cleveref}
		\usepackage{bookmark}
			\bookmarksetup{numbered,open}

\usepackage{oplotsymbl} 
\usepackage{stmaryrd} 
	\numberwithin{figure}{section}
	\numberwithin{equation}{section}
\newtheorem{theorem}{Theorem}[section]
\newtheorem*{theorem*}{Main Theorem}
\newtheorem{lemma}[theorem]{Lemma}
\newtheorem{corollary}[theorem]{Corollary}
\newtheorem*{corollary*}{Corollary}
\newtheorem{proposition}[theorem]{Proposition}
\newtheorem{mainthm}{Theorem}

\newtheorem{mainprop}[mainthm]{Proposition}
    \theoremstyle{definition}
\newtheorem{maindef}[mainthm]{Definition}
\newtheorem{mainsetup}[mainthm]{Setup}

    \theoremstyle{definition}
\newtheorem{definition}[theorem]{Definition}
\newtheorem{setup}[theorem]{Setup}
\newtheorem*{setup*}{Setup}
\newtheorem*{definition*}{Definition}

\newtheorem{remark}[theorem]{Remark}
\newtheorem*{remark*}{Remark}

\newtheorem{construction}[theorem]{Construction}

\newtheorem{example}[theorem]{Example}

\DeclareMathOperator{\Mor}{\mathsf{Mor}}
\def\Ob{\operatorname{Ob}}
\def\Ext{\operatorname{Ext}}

\def\Hom{\operatorname{Hom}}

\def\mod{\operatorname{\mathsf{mod}}}
\def\proj{\operatorname{\mathsf{proj}}}

\def\Ab{\mathsf{Ab}}
\def\add{\operatorname{\mathsf{add}}}

\newcommand{\op}{\mathsf{op}}
\renewcommand{\sp}{\mathsf{sp}}
\newcommand{\A}{\mathscr{A}}

\newcommand{\C}{\mathscr{C}}
\newcommand{\D}{\mathscr{D}}
\newcommand{\E}{\mathcal{E}}

\renewcommand{\H}{\mathcal{H}}

\renewcommand{\L}{\mathcal{L}}
\newcommand{\M}{\mathscr{M}}
\newcommand{\N}{\mathscr{N}}
\renewcommand{\P}{\mathcal{P}}

\newcommand{\R}{\mathcal{R}}
\newcommand{\U}{\mathcal{U}}
\newcommand{\V}{\mathcal{V}}

\newcommand{\X}{\mathscr{X}}

\newcommand{\Ker}{\operatorname{Ker}}
\renewcommand{\Im}{\operatorname{Im}}

\newcommand{\id}{\mathsf{id}}
\newcommand{\xto}{\xrightarrow}
\newcommand{\Sn}{\mathscr{S}_\N}

\def\cone{\operatorname{\mathsf{Cone}}}
\def\cocone{\operatorname{\mathsf{CoCone}}}

\def\isoclass{\operatorname{\mathsf{isoclass}}}

\def\ind{\operatorname{\mathsf{ind}}}
\def\index{\operatorname{\mathsf{index}}}

 \newcommand{\lra}{\longrightarrow}    

\let\oldtocsubsection=\tocsubsection

\renewcommand{\tocsubsection}[2]{\hspace{20pt}\oldtocsubsection{#1}{#2}}
	\newcommand{\deff}{\coloneqq}
	\newcommand{\sse}{\subseteq}
	\newcommand{\rtail}{\rightarrowtail}
	
	\newcommand{\onto}{\twoheadrightarrow}

        \renewcommand{\sec}{\mathsf{sec}}
        \newcommand{\ret}{\mathsf{ret}}
        \newcommand{\inc}{\mathsf{inc}}
	
    \newcommand{\fs}{\mathfrak{s}}
    \newcommand{\ft}{\mathfrak{t}}
    \newcommand{\fu}{\mathfrak{u}}

	\newcommand{\BE}{\mathbb{E}}
	\newcommand{\BF}{\mathbb{F}}

	\newcommand{\BZ}{\mathbb{Z}}

\newcommand{\indpent}{\mathchoice
  			{\pentago}
  			{\pentago}
  			{\scalebox{.7}{\pentago}}
  			{\scalebox{.5}{\pentago}}
		}

	\usetikzlibrary{matrix,arrows,decorations.pathmorphing,positioning,decorations.pathreplacing}
	\tikzset{commutative diagrams/.cd, 
		mysymbol/.style = {start anchor=center, end anchor = center, draw = none}}
    \tikzset{
    labl/.style={anchor=north, rotate=90, inner sep=1mm}
    }

	\let\amph=& 
	
	\tikzcdset{every label/.append style = {font = \footnotesize}}

\newcommand{\commutes}[2][\circ]{\arrow[mysymbol]{#2}[description]{#1}}

\newcommand{\wPB}[1]{\arrow[mysymbol]{#1}[description]{\mathrm{(wPB)}}}

\newcommand{\ol}[1]{\overline{#1}}

\makeatletter

 \def\@setaddresses{\par
     \nobreak \begingroup
     \setstretch{1} 
     \setlength{\parindent}{0cm}
     \footnotesize
     \def\author##1{\nobreak\addvspace\bigskipamount}%
     \def\\{\unskip, \ignorespaces}%
     \interlinepenalty\@M
     \def\address##1##2{\begingroup
         \par\addvspace\bigskipamount
     \@ifnotempty{##1}{(\ignorespaces##1\unskip) }%
     {\scshape\ignorespaces##2}\par\endgroup}%
     \def\curraddr##1##2{\begingroup
     \@ifnotempty{##2}{\nobreak\curraddrname
     \@ifnotempty{##1}{, \ignorespaces##1\unskip}\/:\space
     ##2\par}\endgroup}%
     \def\email##1##2{\begingroup
     \@ifnotempty{##2}{\nobreak\emailaddrname
     \@ifnotempty{##1}{, \ignorespaces##1\unskip}\/:\space
     \ttfamily##2\par}\endgroup}%
     \def\urladdr##1##2{\begingroup
     \def~{\char'\~}%
     \@ifnotempty{##2}{\nobreak\urladdrname
     \@ifnotempty{##1}{, \ignorespaces##1\unskip}\/:\space
     \ttfamily##2\par}\endgroup}%
     \addresses
     \endgroup
 }

\newcommand{\colim@}[2]{%
  \vtop{\m@th\ialign{##\cr
    \hfil$#1\operator@font colim$\hfil\cr
    \noalign{\nointerlineskip\kern1.5\ex@}#2\cr
    \noalign{\nointerlineskip\kern-\ex@}\cr}}%
}
\newcommand{\colim}{%
  \mathop{\mathpalette\colim@{\rightarrowfill@\scriptscriptstyle}}\nmlimits@
}

\makeatother

\begin{document}
\title[An extriangulated resolution theorem and the index]{A resolution theorem for extriangulated categories with applications to the index}

\author[Ogawa]{Yasuaki Ogawa}
	\address{Faculty of Engineering Science, Kansai University, Japan}
	\email{y\underline{ }ogawa@kansai-u.ac.jp} %

\author[Shah]{Amit Shah}
	\address{Department of Mathematics, Aarhus University, Denmark}\
	\email{amit.shah@math.au.dk} %



\begin{abstract}
Quillen's Resolution Theorem in algebraic $K$-theory provides a powerful computational tool for calculating $K$-groups of exact categories. At the level of $K_0$, this result goes back to Grothendieck. In this article, we first establish an extriangulated version of Grothendieck's Resolution Theorem. 

Second, we use this Extriangulated Resolution Theorem to gain new insight into the index theory of triangulated categories. Indeed, we propose an index with respect to an extension-closed subcategory $\mathscr{N}$ of a triangulated category $\mathscr{C}$ and we prove an additivity formula with error term. Our index recovers the index with respect to a contravariantly finite, rigid subcategory $\mathscr{X}$ defined by J{\o}rgensen and the second author, as well as an isomorphism between $K_0^{\mathsf{sp}}(\mathscr{X})$ and the Grothendieck group of a relative extriangulated structure $\mathscr{C}_{R}^{\mathscr{X}}$ on $\mathscr{C}$ when $\mathscr{X}$ is $n$-cluster tilting. In addition, we generalize and enhance some results of Fedele. Our perspective allows us to remove certain restrictions and simplify some arguments. 

Third, as another application of our Extriangulated Resolution Theorem, we show that if $\mathscr{X}$ is $n$-cluster tilting in an abelian category, then the index introduced by Reid gives an isomorphism $K_0(\mathscr{C}_R^{\mathscr{X}}) \cong K_0^{\mathsf{sp}}(\mathscr{X})$.
\end{abstract}
\maketitle
{\small
\tableofcontents
}

\setlength{\baselineskip}{15pt}

\section{Introduction}
\label{sec_intro}

Algebraic $K$-theory has its roots in geometry, drawing its name from Grothendieck's proof of what is now known as the Grothendieck-Riemann-Roch theorem \cite{BorelSerre58}, \cite{SGA6}. The base case of this theory focuses on the Grothendieck group $K_0(\C)$ of a, classically, (skeletally small) abelian or exact category $\C$, which is defined in a purely algebraic fashion. 
The abelian group $K_0(\C)$ is the free group generated on the set of isomorphism classes $[A]$, for $A\in\C$, modulo the relations $[A]-[B]+[C]$ for each admissible short exact sequence $A\rtail B \onto C$ in $\C$. 
However, a priori, computing $K_0(\C)$ may be quite difficult. 
For example, if $R$ is a ring and $\mod R$ is the exact category of finitely generated (left) $R$-modules, 
then determining $K_0(\mod R)$ has led to a rich literature, e.g.\ \cite{HR65,HMW92,Hol15}.

A Resolution Theorem, attributed to Grothendieck (see Bass \cite[Thm.~VIII.8.2]{Bas68}, or \cref{thm_Quillen_resolution_ex}), shows that 
if we can identify a suitable subcategory $\D$ of $\C$, then $K_0(\C)\cong K_0(\D)$. Here, `suitable' means each object in $\C$ admits a finite resolution by objects in $\D$ (see \cref{def_X_resolution}), and $\D$ is an extension-closed subcategory of $\C$ that is also closed under taking kernels of admissible deflations. 
In the case $\C = \mod R$, if every $M\in\mod R$ has finite projective dimension, then we may choose $\D = \proj R$, the subcategory of finitely generated projective $R$-modules. Moreover, since all short exact sequences in $\proj R$ split, the group $K_0(\proj R)$ is in principle simpler to compute.

The first aim of this article is to establish an extriangulated version of the Resolution Theorem (see \cref{ThmB} below).
The notion of an extriangulated category was introduced by Nakaoka--Palu \cite{NP19} and is a unification of Quillen's exact categories and Grothendieck--Verdier's triangulated categories. 
Extriangulated category theory serves as a convenient framework in which to write down proofs that apply to both exact and triangulated categories, and more generally to
their substructures. 
Like the theory of exact categories, but in contrast to the triangulated category theory, each extension-closed subcategory of an extriangulated category inherits an extriangulated structure.
We recall the necessary preliminaries on extriangulated categories in \S\ref{subsec_loc_of_ET}.

Since the theory of extriangulated categories was introduced in 2019, it has allowed many notions and constructions to be streamlined, such as cotorsion pairs \cite{LN19} and Auslander-Reiten theory \cite{INP19}. 
In this same flow, our Extriangulated Resolution Theorem is indeed a generalization of Grothendieck's Resolution Theorem. The Grothendieck group of an extriangulated category is defined in a similar fashion to that of an exact category; see \cref{def:Groth-grou-of-extri-cat}.

\begin{mainthm}[\cref{thm_Quillen_resolution}]
\label{ThmB}
Assume that $\X$ is a full additive subcategory of a skeletally small extriangulated category $(\C,\mathbb{E},\mathfrak{s})$, and that $\X$ is extension-closed and closed under taking cocones of $\mathfrak{s}$-deflations $g\colon B\to C$ where $B,C\in\X$.
If each object $C\in\C$ admits a finite $\X$-resolution, then we have a group isomorphism $
\begin{tikzcd}[column sep=0.6cm,cramped]
    K_0(\C)
    \arrow{r}{\cong}
    &K_0(\X).
\end{tikzcd}$
\end{mainthm}

We also mention here that Quillen founded 
higher algebraic $K$-theory for exact categories in \cite{Qui73}, wherein he defined $K$-groups $K_i(\C)$ for $\C$ exact and all integers $i\geq 0$. 
In addition, Quillen generalized the Resolution Theorem to higher $K$-groups (see \cite[Thm.~3]{Qui73}) and this result is a fundamental theorem in $K$-theory, providing powerful computational machinery. 
There is also an ever-growing number of results on the $K$-theory of triangulated categories, e.g.\ \cite{TT90, Nee05, Sch06, CX14}.

The second aim of this article is to augment the index theory of triangulated categories by exploiting that \cref{ThmB} offers new insight into this theory.  
The index was introduced by Palu \cite{Pal08} in order to better understand the Caldero--Chapoton map, and it continues to find beneficial applications in various contexts. 
For example, in \cite{GNP23} it was revealed that the index is closely related to $0$-Auslander extriangulated categories in connection with various mutations in representation theory, such as cluster tilting mutation \cite{BMRRT, IY08} and silting mutation \cite{AI12, Kim23}.

For a suitable skeletally small $2$-Calabi--Yau triangulated category $\C$ with suspension functor $[1]$ and a $2$-cluster tilting subcategory $\X\sse\C$ (see \cref{example:n-CT-resolution}), Palu defined the \emph{index} of an object $C\in\C$ with respect to $\X$ as a certain element in the split Grothendieck group $K_0^{\sp}(\X)$ of $\X$. The group $K_0^{\sp}(\X)$ is the abelian group freely generated on isomorphism classes $[X]^{\sp}$ of objects $X\in\X$, modulo the relations $[X']^{\sp}-[X]^{\sp}+[X'']^{\sp}$ for each split exact sequence $X'\to X \to X''$ in $\X$. 
In this setup, the index of $C\in\C$ is defined as the element $\ind_\X(C) = [X_0]^{\sp} - [X_1]^{\sp}$ in $K_0^{\sp}(\X$), where $C$ admits a triangle 
$\begin{tikzcd}[column sep=0.5cm]
    X_1 \arrow{r}& X_0 \arrow{r}& C \arrow{r}& X_1[1]
\end{tikzcd}$
in $\C$ with $X_i\in\X$; see \cite[\S2.1]{Pal08} or \cref{def_index}.

Using extriangulated categories, 
Padrol--Palu--Pilaud--Plamondon \cite{PPPP19} observed that this triangle can be viewed as a projective presentation of $C$. 
Indeed, in a certain extriangulated category $(\C,\BE^\X_R,\fs^\X_R)$ \emph{relative} to the triangulated structure on $\C$ (see \S\ref{subsec_rel_theory}), the subcategory $\X$ is precisely the subcategory of all \emph{$\BE^\X_R$-projective} objects (see \cite[Def.~3.23]{NP19}). 
Moreover, 
it follows from \cite[Prop.~4.11]{PPPP19} 
that the index induces the following isomorphism (see also \cref{cor_Palu_index}). 
\begin{equation}\label{PPPP-index-isom-intro}
\begin{tikzcd}
    \ind_\X\colon K_0(\C,\BE^\X_R,\fs^\X_R)\arrow{r}{\cong}& K_0^{\sp}(\X)
\end{tikzcd}
\end{equation}

The isomorphism \eqref{PPPP-index-isom-intro} prompted J{\o}rgensen and the second author to define an index with respect to a more general class of subcategory in \cite{JS23}. 
We note here that the index has been defined for any contravariantly finite, additive, direct-summand-closed subcategory of a skeletally small, idempotent complete triangulated category in \cite{FJS24a}.

Assume the following for the remainder of \S\ref{sec_intro}.

\begin{mainsetup}\label{setup:intro}
    Let $\C$ be a skeletally small, idempotent complete, triangulated category with suspension $[1]$. In addition, suppose $\X\subseteq\C$ is a full, additive subcategory that is contravariantly finite, rigid and closed under direct summands.
\end{mainsetup}

The \emph{index} of $C\in\C$ with respect to $\X$, in the sense of \cite[Def.~3.5]{JS23}, is the class $[C]^\X_R$ in $K_0(\C,\BE^\X_R,\fs^\X_R)$. 
Using this index, the $n$-cluster tilting analogue of \eqref{PPPP-index-isom-intro} was proven in \cite[Thm.~4.10]{JS23}; see \cref{cor_JS_index}. 
In \S\ref{sec_index}, we use \cref{ThmB} to generalize \cite[Thm.~4.10]{JS23}, and hence also \cite[Prop.~4.11]{PPPP19} (i.e.\ \eqref{PPPP-index-isom-intro}); see \cref{cor_Quillen_resolution_to_index}. 
We do this by proposing a new index.

One can set  $\N\deff\X^{\perp_0}\deff\Set{C\in\C | \C(\X,C)=0}$, which is an extension-closed subcategory of $\C$, 
and 
define an extriangulated category 
$(\C,\BE^R_\N,\fs^R_\N)$ (see \cref{prop_relative2}) relative again to the triangulated structure on $\C$. 
We prove in \cref{lem_comparison_relative_structures} that, as extriangulated categories, we have 
$(\C,\BE^R_\N,\fs^R_\N)
    = (\C,\BE^\X_R,\fs^\X_R)
$. 
In particular, we have 
\[
K_0(\C,\BE^R_\N,\fs^R_\N)
    = K_0(\C,\BE^\X_R,\fs^\X_R),
\]
and hence we make the following definition.

\begin{maindef}[\cref{def:index-wrt-extn-closed}]
\label{def:index-intro}
    The \emph{(right) index with respect to $\N$} of an object $C\in\C$ is its class 
$[C]^R_\N$ in $K_{0}(\C,\BE^R_\N,\fs^R_\N)$.
\end{maindef}

We remark that \cref{def:index-intro} makes sense for any extension-closed subcategory $\N$ of $\C$, not just $\X^{\perp_{0}}$, and hence it is a strict generalization of the index $[-]^\X_R$ of \cite{JS23}. 
The upshot of defining an index with respect to an extension-closed subcategory is that we can utilize the localization theory of extriangulated categories as established by Nakaoka, Sakai and the first author \cite{NOS22} (see \S\ref{subsub_localization}). 
In \cite{NOS22} a unification of the Serre quotient and Verdier quotient constructions is given. 
The extriangulated localization theory for triangulated categories is explored further in \cite{Oga22b} and, in particular, there is a certain localization, or quotient,  $\C/\N$ of $\C$ with respect to $\N$ that is abelian; see \eqref{eqn:localization-of-C} and \cref{N-serre-local-is-abelian}. We denote the localization functor by $Q\colon \C\to \C/\N$. Our results in \S\ref{subsec_additive_formula} culminate in the following additivity formula with error term for our index $[-]^R_\N$.

\begin{mainthm}\label{mainthm:additivity}
There is a well-defined group homomorphism 
$\theta^R_\N\colon K_0(\C/\N)\lra K_{0}(\C,\BE^R_\N,\fs^R_\N)$, so that, for any triangle $A\overset{f}{\lra}B\overset{g}{\lra}C\overset{h}{\lra} A[1]$ in $\C$, we have 
$[A]_\N^R-[B]_\N^R+[C]_\N^R=\theta^R_\N\bigl(\Im Q(h)\bigr)$.
\end{mainthm}

We recover the additivity formula with error term for the index $[-]^\X_R$ (see \cite[Thm.~A]{JS23}) as a special case of \cref{mainthm:additivity}; see \cref{rem:connection-to-JS}. 
Moreover, the theory of \cite{NOS22} allows us to simplify some arguments.

The homomorphism $\theta^R_\N$ captures the kernel of the  canonical surjection $\pi^R_{\N}\colon K_0(\C,\BE^R_\N,\fs^R_\N)\to K_0(\C)$, as the next main result shows. This is part of \cref{prop_comparison}.

\begin{mainprop}
\label{mainprop:E}
The sequence
$
\begin{tikzcd}[ampersand replacement=\&,column sep=0.7cm,cramped]
        K_0(\C/\N)
            \arrow{r}{\theta^R_\N}
        \& K_{0}(\C,\BE^R_\N,\fs^R_\N)
            \arrow{r}{\pi^R_{\N}}
        \& K_0(\C)
            \arrow{r} 
        \& 0
\end{tikzcd}
$
is right exact. 
\end{mainprop}

Various indices have also appeared recently in higher homological algebra, e.g.\ \cite{Rei20a,Rei20b,Fed21,Jor21,JS21, FJS24b}.
In \S\ref{subsec_d_angulated_cat} 
we add our index to this landscape. 
We focus again on the $n$-cluster tilting situation,  extending and strengthening some results of Fedele \cite[Thm.~C, Prop.~3.5]{Fed21}, and also a result of Palu \cite[Lem.~9]{Pal09} when $n=2$.
That we can use the right exact sequence in \cref{mainprop:E} to improve and/or strengthen these results adds to the mounting evidence that index theory for triangulated categories should be viewed through the lens of extriangulated categories.

In the last section, we turn our attention to the abelian setting and, using a straightforward application of \cref{ThmB}, deduce an analogue of \eqref{PPPP-index-isom-intro} for an $n$-cluster tilting subcategory in an abelian category; see \cref{thm:abelian-n-CT-isomorphism}. This isomorphism is induced by the index defined by Reid in \cite[Sec.~1]{Rei20b}; see \cref{rem:relation-to-Reid-index}.

\medskip
\noindent
{\bf Notation and conventions.}
All categories and functors in this article are always assumed to be additive, and subcategories will always be full and closed under isomorphisms in the ambient category. 
For a category $\C$, we denote the class of all morphisms in $\C$ by $\Mor\C$, and $\mod\C$ is the category of finitely presented contravariant functors from $\C$ to the abelian category $\Ab$ of abelian groups.
In addition, if $\A\sse\C$ is an additive subcategory, then we denote by $[\A]$ the (two-sided) ideal of morphisms in $\C$ that factor through an object in $\A$. 
The canonical additive quotient functor is denoted $\overline{(-)}\colon \C\to \C/[\A]$, so that 
the image in $\C/[\A]$ of a morphism $f\in\C(A,B)$ is denoted $\overline{f}$.

\section{Extriangulated categories}
\label{subsec_loc_of_ET}
An \emph{extriangulated category} is defined to be an additive category $\C$ equipped with
\begin{enumerate}
\item a biadditive functor $\mathbb{E}\colon\C^{\op}\times\C\to \Ab$, where $\Ab$ is the category of abelian groups, and 
\item a correspondence $\mathfrak{s}$ that associates an equivalence class $\fs(\delta) = [A\overset{f}{\lra} B\overset{g}{\lra} C]$ of a sequence  $A\overset{f}{\lra} B\overset{g}{\lra} C$ in $\C$ to each element $\delta\in\mathbb{E}(C,A)$ for any $A,C\in\C$,
\end{enumerate}
where the triplet $(\C,\mathbb{E},\mathfrak{s})$ satisfies the axioms laid out in \cite[Def.~2.12]{NP19}.
We only recall some terminology and basic properties here, referring the reader to \cite{NP19} for an in-depth treatment. 
For any $A,C\in\C$, two sequences 
$A\overset{f}{\lra} B\overset{g}{\lra} C$ 
and 
$A\overset{f'}{\lra} B'\overset{g'}{\lra} C$
in $\C$ 
are said to be \emph{equivalent} if there exists an isomorphism $b\in\C(B,B')$
such that $f'=b\circ f$ and $g=g'\circ b$.

An extriangulated category $(\C,\mathbb{E},\mathfrak{s})$ is simply denoted by $\C$ if there is no confusion. 
For the remainder of \S\ref{subsec_loc_of_ET}, let $\C = (\C,\BE,\fs)$ be an extriangulated category. 

\begin{definition}
We recall the following terminology. 
\begin{enumerate}
\item
An element $\delta\in\mathbb{E}(C, A)$ is called an \emph{$\mathbb{E}$-extension} for any $A,C\in\C$.
For $a\in\C(A,A')$ and $c\in\C(C',C)$, we write as $a_*\delta\deff\mathbb{E}(C,a)(\delta)$ and $c^*\delta\deff\mathbb{E}(c,A)(\delta)$.
\item
If $\delta\in\BE(C,A)$, then a sequence $A\overset{f}{\lra}B\overset{g}{\lra}C$ with $\fs(\delta) = [A\overset{f}{\lra}B\overset{g}{\lra}C]$ 
is called an $\mathfrak{s}$-\textit{conflation}, and in addition $f$ is called an \emph{$\mathfrak{s}$-inflation} and $g$ an \emph{$\mathfrak{s}$-deflation}.
The pair $\langle A\overset{f}{\lra}B\overset{g}{\lra}C, \delta\rangle$ 
is often denoted by 
$A\overset{f}{\lra}B\overset{g}{\lra}C\overset{\delta}{\dashrightarrow}$ 
and we call it an \emph{$\mathfrak{s}$-triangle}.
\item
A \emph{morphism} of $\mathfrak{s}$-triangles from $\langle A\overset{f}{\lra}B\overset{g}{\lra}C, \delta\rangle$ to $\langle A'\overset{f'}{\lra}B'\overset{g'}{\lra}C', \delta' \rangle$ is a triplet $(a,b,c)$ of morphisms in $\C$ 
with $a_{*}\delta = c^{*}\delta'$ and so that the following diagram commutes. 
\[
\begin{tikzcd}
    A 
        \arrow{r}{f}
        \arrow{d}{a}
    & B 
        \arrow{r}{g}
        \arrow{d}{b}
    & C 
        \arrow[dashed]{r}{\delta}
        \arrow{d}{c}
    & {}\\
    A' 
        \arrow{r}{f'}
    & B'
        \arrow{r}{g'}
    & C' 
        \arrow[dashed]{r}{\delta'}
    & {}
\end{tikzcd}
\]
\end{enumerate}
\end{definition}

We remark that if $A\overset{f}{\lra}B\overset{g}{\lra}C$ is an $\fs$-conflation, then $f$ is a weak kernel of $g$ and $g$ is a weak cokernel of $f$ (see \cite[Prop.~3.3]{NP19}). Recall that a \emph{weak kernel of g} is a morphism $K\overset{k}{\lra}B$ with $gk=0$, and such that any morphism $x\in\C(X,B)$ with $gx=0$ factors (not necessarily uniquely) through $k$. A \emph{weak cokernel} is defined dually. Weak (co)kernels are not necessarily uniquely determined up to isomorphism, unlike (co)kernels.

\begin{definition}
Let $A\overset{f}{\lra} B\overset{g}{\lra} C$ be an $\mathfrak{s}$-conflation.
Then we call $C$ a \emph{cone} of $f$ and put $\cone(f)\deff C$.
Similarly, we denote the object $A$ by $\cocone(g)$ and call it a \emph{cocone} of $g$.
We note that this notation is justified since a cone of $f$ (resp.\ cocone of $g$) is uniquely determined up to isomorphism (see \cite[Rem.~3.10]{NP19}).
For any subcategories $\U$ and $\V$ in $\C$, we denote by $\cone(\V,\U)$ the subcategory consisting of objects $X$ appearing in an $\mathfrak{s}$-conflation $V\lra U\lra X$ with $U\in\U$ and $V\in\V$.
The subcategory $\cocone(\V,\U)$ is defined similarly. 
\end{definition}

Note that if $\U$ and $\V$ are additive, then so are $\cone(\V,\U)$ and $\cocone(\V,\U)$. 
However, the subcategories $\cone(\V,\U)$ and $\cocone(\V,\U)$ are not necessarily closed under direct summands in general.

\begin{definition}\label{def:closed-under-cones}
    A subcategory $\X$ of $\C$ is said to be \emph{closed under taking cones of $\fs$-inflations} if whenever 
$\begin{tikzcd}[column sep=0.5cm]
    A \arrow{r}{f}&B\arrow{r}{g}&C
\end{tikzcd}$
is an $\fs$-conflation with $A,B\in\X$, then we have $C\in\X$.
Being \emph{closed under taking cocones of $\fs$-deflations} is defined dually.
\end{definition}

Triangulated (resp.\ exact) structures on an additive category $\C$ naturally give rise to extriangulated structures on $\C$. 

\begin{example}
\cite[Prop.~3.22]{NP19} 
    Suppose $\C$ is a triangulated category with suspension functor $[1]$ and put 
    $\BE(C,A)\deff\C(C,A[1])$ for $A,C\in\C$. 
    Then there is an extriangulated category $(\C,\BE,\fs)$, 
    where 
    $A\overset{f}{\lra}B\overset{g}{\lra}C\overset{h}{\dashrightarrow}$ 
    is an $\mathfrak{s}$-triangle 
    if and only if 
    $A\overset{f}{\lra}B\overset{g}{\lra}C\overset{h}{\lra}A[1]$
    is a distinguished triangle in $\C$.
    In this case, we say that the extriangulated category/structure $(\C,\BE,\fs)$ \emph{corresponds to a triangulated category}.
\end{example}

\begin{example}
\label{example:exact-is-extri}
\cite[Exam.~2.13]{NP19}
    Suppose $(\A,\E)$ is an exact category. Consider the collection 
    $\BE(C,A)\deff\Set{[A\overset{f}{\lra} B\overset{g}{\lra} C]|A\overset{f}{\lra} B\overset{g}{\lra} C \text{ lies in }\E}$. 
    If this forms a set for all $A,C\in\A$, then there is an extriangulated category $(\A,\BE,\fs)$, 
    where 
    $A\overset{f}{\lra}B\overset{g}{\lra}C\overset{\delta}{\dashrightarrow}$ 
    is an $\mathfrak{s}$-triangle 
    if and only if 
    $A\overset{f}{\lra}B\overset{g}{\lra}C$
    is a conflation in $\E$.
    In this case, we say that the extriangulated category/structure $(\A,\BE,\fs)$ \emph{corresponds to an exact category}. 
    In addition, if $(\A,\E)$ is in fact abelian, then we say $(\A,\BE,\fs)$ \emph{corresponds to an abelian category}. 
    Note that the set-theoretic assumption above is satisfied if $\A$ is skeletally small, or if $\A$ has enough projectives or enough injectives.

\end{example}

One of the advantages of revealing an extriangulated structure lies in the fact that it is closed under certain key operations: 
taking extension-closed subcategories; 
passing to a substructure using relative theory; 
taking certain ideal quotients; 
and localization. 
We now describe the parts of the first two of these operations relevant for our intentions, and localization is treated separately in \S\ref{subsub_localization}.

\subsection{Extension-closed subcategories}
\label{sec:ext-closed-extri-cats}

A subcategory $\N$ of $(\C,\BE,\fs)$ is called \emph{extension-closed} if 
\begin{enumerate}[label=\textup{(\roman*)}]
    \item $\N$ is additive and closed under isomorphisms in $\C$, and 
    \item for any $\fs$-conflation $A\lra B\lra C$, if $A,C\in\N$ then $B\in\N$. 
\end{enumerate}
If the extriangulated structure on $\C$ is understood and no confusion may arise, we say that $\N$ is an extension-closed subcategory of $\C$. 
An extension-closed subcategory of an extriangulated category inherits an extriangulated structure in a canonical way.

\begin{proposition}
\label{prop_ext_closed_ET}
\cite[Rem.~2.18]{NP19}
Suppose $\N$ is an extension-closed subcategory of $\C$.
Define $\mathbb{E}|_\N$ to be the restriction of $\mathbb{E}$ to $\N^{\op}\times \N$, and similarly $\mathfrak{s}|_\N\deff\mathfrak{s}|_{\mathbb{E}|_\N}$. 
Then $(\N,\mathbb{E}|_\N,\mathfrak{s}|_\N)$ is an extriangulated category.
\end{proposition}

\subsection{Relative theory}
\label{subsec_rel_theory}

By a relative extriangulated structure, we mean a ``coarser'' or ``less refined'' structure, analogously to what is usually meant in topology. 
Relative theory for triangulated (resp.\ exact) categories has been considered in various contexts, e.g.\ \cite{Bel00, Kra00} 
(resp.\ \cite{AS93, DRSSK99}).
More recently, relative theory for $n$-exangulated categories ($n\geq1$ an integer) was introduced in \cite{HLN21}. 
A triplet $(\C,\BE,\fs)$ is extriangulated if and only if it is $1$-exangulated \cite[Prop.~4.3]{HLN21}, and thus we can employ this relative theory.
The next result follows from \cite[Prop.~3.16]{HLN21}.

\begin{proposition}
\label{prop_relative_ET}
The following conditions are equivalent for an additive subfunctor $\mathbb{F}\subseteq \mathbb{E}$.
\begin{enumerate}[label={\textup{(\arabic*)}}]
\item
$(\C,\mathbb{F},\mathfrak{s}|_\mathbb{F})$ forms an extriangulated category, where $\mathfrak{s}|_\mathbb{F}$ is the restriction of $\mathfrak{s}$ to $\mathbb{F}$.
\item
$\mathfrak{s}|_\mathbb{F}$-inflations are closed under composition.
\item
$\mathfrak{s}|_\mathbb{F}$-deflations are closed under composition.
\end{enumerate}
\end{proposition}

If an additive subfunctor $\mathbb{F}\sse\BE$ satisfies the equivalent conditions of \cref{prop_relative_ET}, then it is called \emph{closed}. 
Furthermore, we say that the extriangulated structure/category $(\C,\mathbb{F},\mathfrak{s}|_\mathbb{F})$ is \emph{relative to} or \emph{a relative theory of} $(\C,\mathbb{E},\mathfrak{s})$.

For the remainder of \S\ref{subsec_rel_theory}, we assume $(\C,\BE,\fs)$ corresponds to a triangulated category with suspension $[1]$. 
We now recall how a subcategory $\X\sse\C$ determines relative extriangulated structures on $\C$.

\begin{definition}
\label{def_relative1}
For objects $A,C\in\C$, we define subsets of $\mathbb{E}(C,A)$ as follows:
\begin{align*}
\mathbb{E}_L^\X(C,A)&\deff
	\Set{h\in\mathbb{E}(C,A) |  h[-1]\circ x=0\textnormal{\ for all\ }x\colon X\to C[-1]\textnormal{\ with\ }X\in\X }\text{, and}\\
\mathbb{E}_R^\X(C,A)&\deff
\Set{h\in\mathbb{E}(C,A) | h\circ x=0\textnormal{\ for all\ }x\colon X\to C\textnormal{\ with\ }X\in\X}.
\end{align*}
\end{definition}

These give rise to closed subfunctors $\mathbb{E}_L^\X$ and $\mathbb{E}_R^\X$ of $\mathbb{E}$ by \cite[Prop.~3.19]{HLN21}.
Actually, they coincide with the closed subfunctors $\BE_{\X[1]}$ and $\BE_{\X}$ in the notation \cite[Def.~3.18]{HLN21}, respectively. 
In particular, putting $\mathbb{E}^\X\deff\mathbb{E}_L^\X\cap\mathbb{E}_R^\X$, we have three extriangulated substructures  
\[
\C_L^\X\deff(\C,\mathbb{E}_L^\X,\mathfrak{s}_L^\X),\quad 
	\C_R^\X\deff(\C,\mathbb{E}_R^\X,\mathfrak{s}_R^\X)\quad \text{and} \quad
	\C^\X\deff(\C,\mathbb{E}^\X,\mathfrak{s}^\X)
\]
on $\C$, 
which are relative to $(\C,\mathbb{E},\mathfrak{s})$. 
In fact, by \cite[Thm.~2.12]{JS21}, these are extriangulated subcategories of $(\C,\mathbb{E},\mathfrak{s})$ in the sense of \cite[Def.~3.7]{Hau21}.

\section{Localization of extriangulated categories}
\label{subsub_localization}

In the pursuit of unifying Verdier \cite{Ver96} and Serre quotients \cite{Gab62}, the localization theory of extriangulated categories with respect to suitable classes of morphisms was introduced in \cite{NOS22}. 
Since we will not need the theory of \cite{NOS22} in full generality, we only provide \cref{Thm_Mult_Loc}, which follows from the main results of  \cite{NOS22}. 
Furthermore, in \S\ref{subsec_abel_loc_of_tri} and \S\ref{sec_abelian_local_of_tri_cats} we specialize to the case of the localization of a triangulated category by an extension-closed subcategory as investigated in \cite{Oga22b}.

The following notion of an exact functor generalizes the classical ones when both  $(\C,\mathbb{E},\mathfrak{s})$ and $(\C',\mathbb{E}',\mathfrak{s}')$ correspond to exact or triangulated categories.

\begin{definition}
\label{Def_exact_functor}
\label{exact-functor}
\cite[Def.~2.32]{B-TS21} 
Let $(\C,\mathbb{E},\mathfrak{s})$ and $(\C',\mathbb{E}',\mathfrak{s}')$ be extriangulated categories.
An \emph{exact functor} $(F,\phi)\colon (\C,\mathbb{E},\mathfrak{s})\to(\C',\mathbb{E}',\mathfrak{s}')$ is a pair consisting of an additive functor $F\colon \C\to\C'$ and a natural transformation $\phi\colon \mathbb{E}\Rightarrow\mathbb{E}'\circ(F^{\op}\times F)$, 
which satisfies
\(
\mathfrak{s}'(\phi_{C,A}(\delta))
    =[FA\overset{Ff}{\lra}FB\overset{Fg}{\lra}FC]
\)
whenever 
$A\overset{f}{\lra}B\overset{g}{\lra}C\overset{\delta}{\dashrightarrow}$ 
is an $\mathfrak{s}$-triangle. 
\end{definition}

One can compose exact functors in the obvious way to obtain another exact functor; see \cite[Def.~2.11]{NOS22}, also \cite[Lem.~3.19]{BTHSS23}.
Extension-closure and relative theory provide typical examples of exact functors.

\begin{example}
\label{ex_exact_functor}
Let $(\C,\mathbb{E},\mathfrak{s})$ be any extriangulated category. 
\begin{enumerate}[label=\textup{(\arabic*)}]
\item
Let $\N\sse\C$ be extension-closed and consider the extriangulated category $(\N,\mathbb{E}|_\N,\mathfrak{s}|_\N)$; see \cref{prop_ext_closed_ET}. 
The canonical inclusion functor $\inc\colon\N\to\C$ induces an exact functor 
$(\inc,\iota)\colon (\N,\mathbb{E}|_\N,\mathfrak{s}|_\N)\to(\C,\mathbb{E},\mathfrak{s})$, where $\iota\colon \BE|_\N \Rightarrow \BE$ is the canonical inclusion natural transformation.

\item
For a closed subfunctor $\mathbb{F}\subseteq\mathbb{E}$ and the relative extriangulated category $(\C,\mathbb{F},\mathfrak{s}|_\mathbb{F})$, 
the identity $\id_{\C}$ and inclusion $\BF\sse\BE$ constitute an exact functor $(\C,\mathbb{F},\mathfrak{s}|_\mathbb{F})\to(\C,\mathbb{E},\mathfrak{s})$. 

\end{enumerate}
In these situations, both $(\N,\mathbb{E}|_\N,\mathfrak{s}|_\N)$ and $(\C,\mathbb{F},\mathfrak{s}|_\mathbb{F})$ are extriangulated subcategories of $(\C,\mathbb{E},\mathfrak{s})$.
\end{example}

To avoid any set-theoretic problem, we will work under the following setup when considering the localization in the sense of \cite{GZ67}.

\begin{setup}
    We let $(\C,\BE,\fs)$ denote a skeletally small extriangulated category. 
\end{setup}

\cref{Thm_Mult_Loc} recalls sufficient conditions on the pair $(\C,\N)$, where $\N$ is a subcategory of $\C$, to give rise to an extriangulated ``quotient'' category of $\C$ by $\N$. We now lay out the terminology and notation necessary to state this result.

\begin{definition}
\label{def_thick}
\cite[Def.~4.1]{NOS22} 
An additive subcategory $\N$ of $(\C,\BE,\fs)$ is called \emph{thick} if it is closed under direct summands, and $\N$ satisfies the \emph{$2$-out-of-$3$ property} for $\mathfrak{s}$-conflations, that is, for any $\mathfrak{s}$-conflation $A\lra B\lra C$, if any two of $A, B$ or $C$ belong to $\N$, then so does the third.
\end{definition}

Notice that any thick subcategory $\N$ of $\C$ is automatically extension-closed by definition. 
In the case that $(\C,\BE,\fs)$ corresponds to a triangulated category, the notion of a thick subcategory as in \cref{def_thick} coincides with the usual one for triangulated categories.

In addition, a thick subcategory $\N\sse\C$ is said to be \emph{Serre} if whenever $A\lra B\lra C$ is an $\mathfrak{s}$-conflation with $B\in\N$, then we have $A,C\in\N$ (see \cite[Def.~1.17]{Oga22b}).
This generalizes the notion of a Serre subcategory for exact categories introduced in \cite[4.0.35]{C-E98}.

\begin{definition}
\label{def_Sn_from_thick}
\cite[Def.~4.3]{NOS22} 
We associate the following classes of morphisms to a thick subcategory $\N\sse\C$:
\begin{enumerate}
\item
$\L \deff \Set{f\in\Mor\C | f \ \textnormal{is an $\mathfrak{s}$-inflation with}\ \cone(f)\in\N}$; 
\item
$\R \deff \Set{g\in\Mor\C | g \ \textnormal{is an $\mathfrak{s}$-deflation with}\ \cocone(g)\in\N}$; and

\item $\Sn$ is the smallest subclass of $\Mor\C$ closed under compositions containing both $\L$ and $\R$.
\end{enumerate}
\end{definition}

By \cite[p.~374]{NOS22}, the class $\Sn$ consists of all finite compositions of morphisms in $\L$ and $\R$ and, 
moreover, it satisfies the following condition.
\begin{enumerate}[label=\textup{(M\arabic*)}]
  \setcounter{enumi}{-1}

	\item\label{MS0}
	$\Sn$ contains all isomorphisms in $\C$, and is closed under compositions and taking finite direct sums.
\end{enumerate}

If $\N\sse\C$ is thick, then prototypical examples of localizing $\C$ at the class $\Sn$ are Verdier and Serre quotients of triangulated and abelian categories, respectively; see \cite[\S4.2]{NOS22}.

\begin{example}
\label{ex_list}
	\begin{enumerate}
		\item (Verdier quotient.) \cite[Exam.~4.8]{NOS22} 
		Let $(\C,\BE,\fs)$ be a triangulated category and $\N$ a thick subcategory of $\C$.
		The class $\Sn$ satisfies $\Sn=\L=\R$ and 
		the localization $\C[\Sn^{-1}]$ is just the usual Verdier quotient of $\C$ with respect to $\N$.

		\item (Serre quotient.) \cite[Exam.~4.9]{NOS22} 
		Let $\N$ be a Serre subcategory of an abelian category $(\C,\BE,\fs)$. 
		Then we have $\Sn=\L\circ\R$ and $\C[\Sn^{-1}]$ is the usual Serre quotient of $\C$ with respect to $\N$. 
	\end{enumerate}
\end{example}

In contrast to the triangulated and abelian cases, it is not clear if the localization $\C[\Sn^{-1}]$ is equipped with a natural extriangulated structure in general. 
However, sufficient conditions 
for this are identified in \cite[p.~343]{NOS22}, namely, conditions (MR1)--(MR4) concerning $\Sn$.
Since we will not need these conditions explicitly, we omit recalling them here.

Consider the localization $L\colon \C \to \C[\Sn^{-1}]$. 
The class $\Sn$ is said to be \emph{saturated} if, for any $f\in\Mor\C$, we have $L(f)$ is an isomorphism if and only if $f\in\Sn$. 
Let $\overline{(-)}\colon \C\to \overline{\C}$ denote the quotient functor, where $\overline{\C}\deff \C/[\N]$ is the additive ideal quotient. 
We put $\overline{\Sn}\deff\Set{\overline{f} | f\in\Sn}$ accordingly. 
Note that $\overline{f}=0$ if and only if $f$ factors through an object in $\N$.
The localization of $\overline{\C}$ at $\overline{\Sn}$ is denoted by 
\begin{equation}\label{eqn:localization-of-C}
\C/\N\deff 
	\overline{\C}[\overline{\Sn}^{-1}]
\end{equation}
and the localization functor by $\overline{Q}\colon \overline{\C}\to\C/\N$. 
We define $Q\colon \C \to \C/\N$ to be the composition $\overline{Q} \circ \overline{(-)}$. 
Note $Q$ factors uniquely through $L$:
\begin{equation}\label{eqn:functor-M}
\begin{tikzcd}[column sep=1.5cm]
\C 
	\arrow{r}{L}
	\arrow{d}[swap]{\overline{(-)}}
	\arrow{dr}[xshift=0.2cm, yshift=-0.2cm]{Q}
& \C[\Sn^{-1}]
	\arrow[dotted]{d}{^{\exists !}M}
\\
\ol{\C} 
	\arrow{r}{\ol{Q}}
& \C/\N.
\end{tikzcd}
\end{equation}

\begin{lemma}\label{lem:GZ-localization-isomorphisms}
If $\C$ is weakly idempotent complete, then $\C[\Sn^{-1}] \cong \ol{\C}[\ol{\Sn}^{-1} ]$.
\end{lemma}

\begin{proof}
Consider $\mathsf{S} \deff \Set{ f\in\Mor\C | \text{$f$ is a section and admits a cokernel in $\N$} }$. One can check that $\ol{\C}$ is isomorphic to the localization $\C[\mathsf{S}^{-1}]$; see \cite[Exam.~2.6]{Oga22a}. 
Note that since $\C$ is weakly idempotent complete, we have $\mathsf{S} \sse \L \sse \Sn$ (see e.g.\ \cite[Prop.~2.5]{BHST22}). 
It follows that $L = P\circ \ol{(-)}$ and $P = N\ol{Q}$ for uniquely induced functors $P\colon \ol{\C}\cong \C[\mathsf{S}^{-1}] \to \C[\Sn^{-1}]$ and $N \colon \C/\N \to \C[\Sn^{-1}]$. One can then check that $N$ is an inverse of the functor $M$ from \eqref{eqn:functor-M}.
\end{proof}

Thus, if $\C$ is weakly idempotent complete (e.g.\ a triangulated category), we may think of the functor $Q \colon \C \to \C/\N$ as the localization of $\C$ at $\Sn$.

\begin{theorem}
\label{Thm_Mult_Loc}
Let $(\C,\BE,\fs)$ be a skeletally small extriangulated category with a thick subcategory $\N$. 
Suppose $\Sn$ is saturated and $\ol{\Sn}$ satisfies conditions {\rm (MR1)--(MR4)} of \cite{NOS22}. 
Then there is an extriangulated category $(\C/\N,\widetilde{\mathbb{E}},\widetilde{\mathfrak{s}})$ together with an (appropriately universal) exact functor $(Q,\mu)\colon (\C,\BE,\fs)\to (\C/\N,\widetilde{\BE},\widetilde{\fs})$ satisfying $\Ker Q=\N$.
In particular, we obtain a sequence
\begin{equation}\label{seq_extriangulated_localization}
\begin{tikzcd}
    (\N,\BE|_\N,\fs|_\N)
        \arrow{r}{(\inc,\iota)}
    & (\C,\BE,\fs)
        \arrow{r}{(Q,\mu)}
    & (\C/\N,\widetilde{\BE},\widetilde{\fs})
\end{tikzcd}
\end{equation}
of exact functors.
\end{theorem}
\begin{proof}
The exact functor $(Q,\mu)\colon (\C,\BE,\fs)\to (\C/\N,\widetilde{\BE},\widetilde{\fs})$ is constructed in \cite[Thm.~3.5(1)]{NOS22} and 
its universality is shown in \cite[Thm.~3.5(2)]{NOS22} (see also \cite[Prop.~4.3]{ES22}).
The equality $\Ker Q=\N$ follows from \cite[Lem.~4.5]{NOS22}.
\end{proof}

\begin{definition}
\label{def:extri-localization}
In the setup of \cref{Thm_Mult_Loc}, 
we call the exact functor 
$(Q,\mu)\colon (\C,\mathbb{E},\mathfrak{s})
	\to (\C/\N,\widetilde{\mathbb{E}},\widetilde{\mathfrak{s}})$
the 
\emph{extriangulated localization of $\C$ with respect to $\N$}.
If there is no confusion, we simply denote it by $Q\colon \C\to\C/\N$.
\end{definition}

We refer to \cite[\S3]{NOS22} for explicit descriptions of $\widetilde{\mathbb{E}}$ and $\widetilde{\mathfrak{s}}$.
However, we note that, by construction, any $\widetilde{\mathfrak{s}}$-inflation (resp.\ $\widetilde{\mathfrak{s}}$-deflation) comes from an $\mathfrak{s}$-inflation (resp.\ $\mathfrak{s}$-deflation); see \cite[Lem.~3.32]{NOS22}.

\subsection{Localization of triangulated categories}
\label{subsec_abel_loc_of_tri}

We now specialize to the case when $(\C,\BE,\fs)$ corresponds to a triangulated category and recall the relevant localization theory.

\begin{setup}\label{setup:localization-of-tri-cats}
We fix a skeletally small, triangulated category $\C$ with suspension $[1]$ and an 
extension-closed 
subcategory $\N\sse\C$ 
that is closed under
direct summands. 
We denote by $(\C,\BE,\fs)$ the extriangulated category corresponding to the triangulated category $\C$.   
\end{setup}

Since $\N$ is extension-closed in $(\C,\BE,\fs)$, it is immediate that it is extension-closed in any extriangulated substructure $(\C,\BF,\fs|_{\BF})$ of $(\C,\BE,\fs)$.
In particular, it is extension-closed in the following relative extriangulated structures defined using $\N$. These differ to those defined in \cref{def_relative1}, but we make a comparison of these structures in a special case in \cref{lem_comparison_relative_structures}.

\begin{proposition}\label{prop_relative2}
\cite[Prop.~2.1]{Oga22b}
For $A,C\in\C$, define subsets of $\BE(C,A) = \C(C,A[1])$ as follows.
\begin{align*}
\mathbb{E}^L_\N(C,A) &\deff \Set{ h\colon C\to A[1] | \forall x\colon N\to C \text{ with } N\in\N\text{, we have } hx\in[\,\N[1]\,] }\\
\mathbb{E}^R_\N(C,A) &\deff \Set{ h\colon C\to A[1] | \forall y\colon A\to N \text{ with } N\in\N\text{, we have } y\circ h[-1]\in[\,\N[-1]\,] } 
\end{align*}
These give rise to closed subfunctors $\mathbb{E}^L_\N$ and $\mathbb{E}^R_\N$ of $\mathbb{E}$.
In particular, putting $\mathbb{E}_\N\deff\mathbb{E}^L_\N\cap\mathbb{E}^R_\N$, we obtain extriangulated structures
\[
\C^L_\N\deff(\C,\mathbb{E}^L_\N,\mathfrak{s}^L_\N),\quad \C^R_\N\deff(\C,\mathbb{E}^R_\N,\mathfrak{s}^R_\N),\quad \C_\N\deff(\C,\mathbb{E}_\N,\mathfrak{s}_\N),
\]
all relative to the triangulated structure $(\C,\mathbb{E},\mathfrak{s})$.
\end{proposition}

We remark that the above structures are generalized in \cite[Prop.~A.4]{Che23} but from the viewpoint of constructing exact substructures of an exact category.
With respect to the relative structure $\C_\N$, the pair $(\C,\N)$ yields a saturated class $\Sn$ of morphisms in $\C$ with $\ol{\Sn}$ satisfying the needed conditions {\rm (MR1)}--{\rm (MR4)} to obtain an extriangulated localization.

\begin{theorem}
\label{thm_abel_loc1}
\cite[Thm.~A, Lem.~2.5, Lem.~2.12]{Oga22b}
The following statements hold.
\begin{enumerate}[label={\textup{(\arabic*)}}]

\item 
\label{lem_coresolving_subcat}
The subcategory $\N\sse\C$ is closed under taking cones of $\mathfrak{s}_\N^R$-inflations. 
Dually, $\N$ is closed under taking cocones of $\mathfrak{s}_\N^L$-deflations. 

\item
\label{N-thick}
The subcategory $\N$ is thick in  $\C_\N$. 
The corresponding class $\Sn\sse\Mor\C$ (see \textup{\cref{def_Sn_from_thick}}) is saturated and $\ol{\Sn}$ satisfies {\rm (MR1)--(MR4)} with respect to $\C_\N = (\C,\BE_{\N},\fs_{\N})$. 
Moreover, $\Sn=\R_{\ret}\circ\L=\R\circ\L_{\sec}$ holds, where 
$\L_{\sec}$ denotes the class of sections belonging to $\L$ 
and 
$\R_{\ret}$ denotes the class of retractions belonging to $\R$.

\item
\label{extri-local}
There exists an extriangulated localization $(Q,\mu)\colon \C_\N\to (\C/\N,\widetilde{\BE_\N},\widetilde{\fs_\N})$ with $\Ker Q=\N$.

\end{enumerate}
\end{theorem}
\begin{proof}
\ref{lem_coresolving_subcat}:\;\;
    Suppose 
    $\begin{tikzcd}[column sep=0.5cm]
        A \arrow{r}{f}& B\arrow{r}{g}& C \arrow[dashed]{r}{h}&{}
    \end{tikzcd}$
    is an $\mathfrak{s}_\N^R$-triangle with $A,B\in\N$. 
    It follows that $h\in[\N]$ by \cref{prop_relative2}. Then \cite[Lem.~2.5(2)]{Oga22b} implies $C\in\N$. The second assertion is proved dually.

\ref{N-thick} and \ref{extri-local}:\;\; This is a combination of \cite[Cor.~2.8, Prop.~2.11, Prop.~2.17, Lem.~2.12, Thm.~2.20]{Oga22b}. 
\end{proof}


\subsection{Abelian localization of triangulated categories}
\label{sec_abelian_local_of_tri_cats}

Here, we review localizations of triangulated categories that are abelian. 
Abelian localizations of triangulated categories can arguably be traced back to hearts of t-structures in the sense of \cite{BBD}. Since then, abelian localizations have been found using cluster tilting subcategories \cite{BMR07,KR07, KZ08}. These constructions were unified in \cite{Nak11, AN12} and placed in an extriangulated context in \cite{LN19}. 
A generalization from cluster tilting to rigid subcategories was initiated in \cite{BM12,BM13}, and has been further developed in \cite{Bel13,Nak13,HS20}. See \cref{ex_BM_BBD} for some details.

\begin{setup}\label{setup:abelian-local}
In addition to \cref{setup:localization-of-tri-cats}, we assume further that $\cone(\N,\N)=\C$ holds in \S\ref{sec_abelian_local_of_tri_cats}.
\end{setup}

Note that when $(\C,\BE,\fs)$ corresponds to a triangulated category (as we are currently assuming), we have $\cone(\N,\N)=\C$ if and only if $\cocone(\N,\N)=\C$. 
In this case, the bifunctors $\mathbb{E}_\N^L$ and $\mathbb{E}_\N^R$ can be described more simply. 
The next observation follows from the proof of \cite[Lem.~4.1]{Oga22b}.

\begin{lemma}
\label{lem_specific_relative_structures}
We have the following identities.
\begin{align*}
\mathbb{E}_\N^L(C,A)&= \Set{h\in\mathbb{E}(C,A) | h[-1]\textnormal{\ factors through an object in\ }\N}\\
\mathbb{E}_\N^R(C,A)&= \Set{h\in\mathbb{E}(C,A) | h\textnormal{\ factors through an object in\ }\N}
\end{align*}
\end{lemma}

Relative structures like the above can be obtained from rigid subcategories.
The next result clarifies how these structures relate to those in \cref{def_relative1} 
in case $\N$ is the kernel of $\C(\X,-)$ for a contravariantly finite, rigid subcategory $\X\sse\C$
(see \cite[Exam.~2.4]{Oga22b}). 
We will use this in \S\ref{subsec_additive_formula} to produce a generalization of 
the index defined in \cite{JS23}. 
Recall that $\X\sse\C$ is \emph{rigid} if $\BE(\X,\X) =\C(\X,\X[1]) = 0$. 

\begin{lemma}
\label{lem_comparison_relative_structures}
Let $\X\subseteq \C$ be a subcategory and consider the extension-closed subcategory 
$\M\deff\X^{\perp_0}\deff\Set{C\in\C | \C(\X,C)=0}$
of $\C$. 
If $\X$ is contravariantly finite and rigid, then we have 
$\mathbb{E}^R_\M(C,A)=\mathbb{E}_R^\X(C,A)$ and 
$\mathbb{E}^L_\M(C,A)=\mathbb{E}_L^\X(C,A)$ for any $A,C\in\C$. 
In particular, 
$\C_\M^R = \C_R^\X$, 
$\C_\M^L = \C_L^\X$ and 
$\C_\M = \C^\X$.
\end{lemma}
\begin{proof}
We only check the first equation. 
Note that $\cocone(\M,\M)=\C$ holds. Indeed, any object $C\in\C$ admits a right $\X$-approximation $x\colon X\to C$ which yields a triangle 
\begin{equation}\label{eqn:lem1-20-X-approx}
    \begin{tikzcd}
        X \arrow{r}{x}& C \arrow{r}{y}& M \arrow{r}{}& X[1]
    \end{tikzcd}
\end{equation}
with $M, X[1]\in\M$.

It is clear that $\mathbb{E}_\M^R(C,A)\sse\mathbb{E}_R^\X(C,A)$. 
Conversely, let $h\in\BE_R^\X(C,A)$ and consider the triangle \eqref{eqn:lem1-20-X-approx} as above. Then the composite $hx$ vanishes, whence $h$ factors through $y$, and so $h\in[ \M ]$. The claim follows from the second identity of \cref{lem_specific_relative_structures}.
\end{proof}

The main result we recall in \S\ref{sec_abelian_local_of_tri_cats} is the following, which establishes the \emph{abelian localization} of a triangulated category with respect to an extension-closed subcategory.

\begin{theorem}
\label{N-serre-local-is-abelian}
\cite[Thm.~4.2, Cor.~4.3]{Oga22b}
The subcategory $\N$ is Serre 
in
$\C_\N$ and the localization $(\C/\N,\widetilde{\BE_\N},\widetilde{\fs_\N})$ corresponds to an abelian category.
Furthermore, the functor 
$Q\colon \C \to \C/\N$ 
is cohomological.
\end{theorem}

We include a simple example to demonstrate that even under very nice assumptions, we cannot hope that $\N$ is a Serre subcategory of $\C^R_\N$.

\begin{example}
\label{example:not-Serre-in-CR}
    Let $\C$ denote the cluster category associated to the quiver $1\to 2$ (in the sense of \cite{BMRRT}) and consider the cluster tilting subcategory $\N\deff\add(2\oplus \substack{1\\2})\sse\C$. Since $\cone(\N,\N) = \C$ is satisfied, we have $\mathbb{E}_\N^R(C,A) = [\N](C,A)$ by \cref{lem_specific_relative_structures}. Now consider the triangle 
    $1[-1] \to 2\to \substack{1\\2}\overset{h}{\to} 1$ in $\C$. Since $\substack{1\\2}\in\N$, we see that $1[-1] \to 2\to \substack{1\\2}$ is an $\fs^R_\N$-conflation. However, $1[-1]$ does not lie in $\N$. Thus, $\N$ is not thick in $\C^R_\N$ and so certainly cannot be Serre.
\end{example}

Despite this example, we need to understand how $Q$ acts on $\fs^R_\N$-conflations in order to connect this viewpoint to the aforementioned index. 
It turns out that $Q$ sends $\fs^R_\N$-conflations to right exact sequences in $\C/\N$. 

\begin{definition}\label{def:left-right-exact-functors}
\cite[Def.~2.7]{Oga21} 
    Suppose $(\D,\BF,\ft)$ is an extriangulated category and $\A$ is an abelian category. A covariant additive functor $F\colon \D \to \A$ is \emph{right exact} if, for every $\ft$-conflation
    $\begin{tikzcd}[column sep=0.5cm]
        A \arrow{r}{f}& B \arrow{r}{g}& C,
    \end{tikzcd}$
    the sequence 
    $\begin{tikzcd}[column sep=0.6cm]
        FA \arrow{r}{Ff}& FB \arrow{r}{Fg}& FC \arrow{r}& 0 
    \end{tikzcd}$
    is exact in $\A$. 
    \emph{Left exact} functors are defined dually.
\end{definition}

In the setup of \cref{def:left-right-exact-functors}, suppose that $\Ext_\A^1(C,A)$ is a set for all $A,C\in\A$, and equip $\A$ with its canonical extriangulated structure $(\A,\Ext_\A^1,\fu)$. 
Then the functor $F$ is both left and right exact (as just defined), if and only if $F$ forms part of an exact functor 
$(\D,\BF,\ft)\to(\A,\Ext_\A^1,\fu)$ in the sense of \cref{Def_exact_functor}.
In this language we thus have:

\begin{corollary}
\label{cor_abel_loc2}
\cite[Cor.~4.4]{Oga22b}
The functor 
$Q\colon\C\to\C/\N$ 
induces a right (resp.\ left) exact functor 
$Q\colon \C^R_\N\to\C/\N$ 
(resp.\ 
$Q\colon \C^L_\N\to\C/\N$).
\begin{equation}\label{diag_one_sided_localization}
\begin{tikzcd}[column sep=1.5cm, row sep=0.2cm]
    {}&\C_{\N}^{R}\arrow[hook]{dl}\arrow[bend left]{drr}{\textnormal{right exact}}[swap]{Q}&{}&{}\\
    (\C,\BE,\fs)&{}&\C_{\N}\arrow[hook]{ul}\arrow[hook]{dl}\arrow{r}{Q}[swap]{\textnormal{exact}}&(\C/\N,\widetilde{\BE_\N},\widetilde{\fs_\N})\\
    {}&\C_{\N}^{L}\arrow[hook]{ul}\arrow[bend right]{urr}{Q}[swap]{\textnormal{left exact}}&{}&{}
\end{tikzcd}
\end{equation}
\end{corollary}

We end this subsection by giving examples of such abelian localizations.

\begin{example}
\label{ex_BM_BBD}
Using the theory developed in \cite{Oga22b}, we can obtain more information about the localization functor in the situations  considered in \cite{BBD} and \cite{BM12,BM13}.
The triplet $(\C,\BE,\fs)$ still denotes a skeletally small triangulated category (see \cref{setup:localization-of-tri-cats}).
\begin{enumerate}[label={\textup{(\arabic*)}}]

\item\label{t-structure} Let $(\U,\V)$ be a $t$-structure on $\C$, namely, a cotorsion pair (see \cite[Def.~2.1]{Nak11}) 
with $\U[1] \subseteq\U$.
There is a cohomological functor $H\colon \C\to\H$ to the abelian heart of $(\U,\V)$ and we put $\N\deff\Ker H$.
Then \cite[Thm.~5.8]{Oga22b} and \cref{cor_abel_loc2} tell us that $H$ induces a right exact functor $H\colon \C^R_\N\to\C/\N$ and a left exact functor $H\colon \C^L_\N\to\C/\N$.

We remark that the assertion still holds for the general heart construction of Abe--Nakaoka \cite{Nak11, AN12}.

\item\label{example:rigid}
Let $\X\sse\C$ be an additive, contravariantly finite, rigid subcategory that is closed under isomorphisms and direct summands, and 
put $\N\deff\X^{\perp_0}$. 
Then $\cone(\N,\N)=\C$ holds by \cref{lem_comparison_relative_structures}, 
and we have a natural exact equivalence $G\colon \C/\N\overset{\simeq}{\lra}\mod\X$ of abelian categories by universality as below; see \cite[Sec.~5.3.2]{Oga22b}.
\[
\begin{tikzcd}
    \N \arrow{r}& \C_{\N} \arrow{r}{Q}\arrow{d}[swap]{F(-)\deff\C(?,-)|_\X}& \C/\N \arrow[dotted]{dl}{\simeq}[swap]{G}\\
    & \mod\X &
\end{tikzcd}
\]
The functor $F$ is the restriction to $\X$ of the Yoneda embedding, namely, for $C\in\C$ 
we put $F(C)\deff \C(?,C)|_\X\in\mod\X$. 
In this case, by \cref{cor_abel_loc2}, the exact functor $Q$ induces a right exact functor $Q\colon \C^R_\N \to \C/\N$ and a left exact functor $Q\colon \C^L_\N\to \C/\N$.
\end{enumerate}
\end{example}

\section{An extriangulated resolution theorem}\label{sec_Quillen_resolution}

The aim of this section is to prove \cref{ThmB}(=\cref{thm_Quillen_resolution}), which is an extriangulated version of a resolution theorem for exact categories (see \cref{thm_Quillen_resolution_ex}).
We will use \cref{thm_Quillen_resolution} in \S\ref{sec_index} to investigate the relationship between Grothendieck groups arising from a triangulated category $(\C,\BE,\fs)$ and an $n$-cluster tilting subcategory $\X\sse\C$. In this case, we know each object in $\C$ has a finite $\X$-resolution (see \cref{example:n-CT-resolution}), but it is not necessarily true that $\X$ is closed under taking cocones of $\fs$-deflations since any morphism is an $\fs$-deflation. Thus, we must pass to a relative extriangulated structure on $\C$; see \cref{cor_Quillen_resolution_to_index} for details.

Although Quillen produced a resolution theorem in the framework of higher algebraic $K$-theory in \cite[\S4]{Qui73}, the idea was first established by Grothendieck at the level of $K_0$ (see \cite[Ch.~VIII, Thm.~4.2]{Bas68}, or \cite[Ch.~II, Thm. 7.6]{Wei13}). 
Let us begin by recalling the definition of the Grothendieck group for an extriangulated category. 
Recall that, for a skeletally small additive category $\C$, the \emph{split Grothendieck group $K_0^{\sp}(\C)$ of $\C$} is the free abelian group generated on the set of isomorphism classes $[A]$ for $A\in\C$, modulo the relations 
$[A] - [B] + [C]$ for each split exact sequence 
$\begin{tikzcd}[column sep=0.5cm, cramped]
    A \arrow{r}& B \arrow{r}& C
\end{tikzcd}$
in $\C$.

\begin{definition}\label{def:Groth-grou-of-extri-cat}
\cite[\S4]{ZZ21} 
    Let $(\C,\BE,\fs)$ be a skeletally small extriangulated category. 
    The \emph{Grothen\-dieck group} of $(\C,\BE,\fs)$ is defined to be 
\[
K_0(\C,\BE,\fs) 
	\deff 
        K_0^{\sp}(\C)  /  \braket{\, [A] - [B] + [C] \, |\begin{tikzcd}[column sep=0.5cm, ampersand replacement=\&]
			A \arrow{r}\& B \arrow{r}\& C
			\end{tikzcd} 
			\text{ is an $\fs$-conflation}\, } .
\]
If it will cause no confusion, we will abbreviate $K_0(\C,\BE,\fs)$ as $K_0(\C)$.
\end{definition}

To recall Grothendieck's Resolution Theorem, we need the following notion.

\begin{definition}
\label{def_X_resolution}
Let $(\C,\mathbb{E},\mathfrak{s})$ be an extriangulated category, let $\X\sse\C$ be a subcategory and fix an object $C\in\C$.
A \emph{finite $\X$-resolution (in $(\C,\mathbb{E},\mathfrak{s})$)} of $C$ is defined to be a complex
\begin{equation}
\label{seq_X_resolution}
\begin{tikzcd}
    X_n 
        \arrow{r}{f_{n-1}}
    & \cdots 
        \arrow{r}{f_1 g_2}
    & X_1 
        \arrow{r}{f_0 g_1}
    & X_0 
        \arrow{r}{g_0}
    & C,
\end{tikzcd}
\end{equation}
where $X_i\in\X$ for each $0\leq i \leq n$, and 
$C_{i+1}\overset{f_i}{\lra} X_i\overset{g_i}{\lra} C_i$ 
is an $\fs$-conflation 
for each $0\leq i \leq n-1$ with $(C_0, C_n)\deff(C, X_n)$.
In this case, we say that the $\X$-resolution \eqref{seq_X_resolution} is of \emph{length} $n$.
In particular, any object $X\in\X$ has an $\X$-resolution of length $0$.
The notion of \emph{finite $\X$-coresolution} is defined dually.
\end{definition}

Typical examples of such resolutions arise from cluster tilting theory.

\begin{example}
\label{example:n-CT-resolution}
Suppose $\C$ is an idempotent complete triangulated category with suspension $[1]$ and let $n\geq 2$ be an integer. 
Recall from \cite[\S3]{IY08} that a subcategory $\X\subseteq \C$ is called an \emph{$n$-cluster tilting subcategory} of $\C$ if $\X$ is functorially finite in $\C$ and 
\begin{align*}
    \X  &= \Set{C\in\C | \C(\X,C[i])=0\text{ for all } 1 \leq i \leq n-1}\\
        &= \Set{C\in\C | \C(C,\X[i])=0\text{ for all } 1 \leq i \leq n-1}.
\end{align*}  
Recall that, for subcategories $\U$ and $\V$ of $\C$, the subcategory $\U*\V$ of $\C$ consists of objects $C$ appearing in a triangle $U\lra C\lra V\lra U[1]$ with $U\in\U$ and $V\in\V$, and that this operation is associative \cite[p.~123]{IY08}.
If $\X$ is $n$-cluster tilting, then 
$(\X,\X*\cdots *\X[n-2])$ is a cotorsion pair \cite[Prop.~3.2, Thm.~3.4]{Bel15} and 
$\C = \X * \X[1] * \X[2] * \cdots * \X[n-1]$  \cite[Thm.~5.3]{Bel15}.
This implies that any $C\in\C$ admits an $\X$-resolution in the relative extriangulated category $\C^\X_R$($=\C_\N^R$ by \cref{lem_comparison_relative_structures} where $\N = \X^{\perp_{0}}$) 
of length at most $n-1$ (see \cite[Cor.~3.3]{IY08}, or \cite[Prop.~3.2, Thm.~3.4]{Bel15}).
\end{example}

\begin{theorem}[Resolution Theorem]
\label{thm_Quillen_resolution_ex}
\cite[Ch.~VIII, Thm.~4.2]{Bas68}
Let $\C$ be a skeletally small exact category.
Assume that $\X\sse\C$ is extension-closed and closed under 
taking kernels of admissible deflations in $\C$.
If any object $C\in\C$ admits a finite $\X$-resolution, then we have $K_0(\X)\cong K_0(\C)$.
\end{theorem}

A typical example of the resolution theorem is as follows. 
For an abelian category $\C$ with enough projectives, if each object in $\C$ has finite projective dimension, then we have $K_0(\P)\cong K_0(\C)$ for $\P\sse\C$ the subcategory of projectives.
Note that $K_0(\P)$ is same as the split Grothendieck group $K_0^{\mathsf{sp}}(\P)$ of $\P$.

The following is an extriangulated version of the classical resolution theorem.
The dual of \cref{thm_Quillen_resolution} also holds.

\begin{theorem}[Extriangulated Resolution Theorem]
\label{thm_Quillen_resolution}
Let $(\C,\BE,\fs)$ be a skeletally small  extriangulated category. 
Suppose $\X$ is an extension-closed subcategory of $(\C,\BE,\fs)$, such that 
$\X$ is closed under taking cocones of $\mathfrak{s}$-deflations. 
If any object $C\in\C$ admits a finite $\X$-resolution, 
then we have an isomorphism
\begin{align*}
    K_0(\C,\BE,\fs) 
    &\overset{\cong}{\longleftrightarrow}
    K_0(\X,\BE|_{\X},\fs|_{\X})
    \\ 
    [C]
    &\longmapsto \sum_{i=0}^{n}(-1)^{i} [X_i]\\
    [X]
    &\longmapsfrom [X],
\end{align*}
where \eqref{seq_X_resolution} is an $\X$-resolution of $C\in\C$.
\end{theorem}

\subsection{The proof of \texorpdfstring{\cref{thm_Quillen_resolution}}{Theorem~4.4}}
\label{proof-of-extriangulated-res-thm}

In this subsection, we work under the hypotheses of \cref{thm_Quillen_resolution}. 

\begin{setup}
Suppose $(\C,\BE,\fs)$ is a skeletally small extriangulated category and $\X\sse\C$ is an extension-closed subcategory that is closed under taking cocones of $\fs$-deflations, such that every object in $\C$ has a finite $\X$-resolution.
\end{setup}

Put $\X_0\deff\X$ and, for any $i>0$, we denote by $\X_i$ the subcategory of $\C$ consisting objects which admit finite $\X$-resolutions of length at most $i$.
Note that $\X_{i} = \cone(\X_{i-1},\X)$ and it is additive and closed under isomorphisms. 
Furthermore, we have an ascending chain $\X_0\subseteq \X_1\subseteq \X_2\subseteq \cdots$, and the union $\bigcup_{i\geq 0}\X_i$ coincides with $\C$ because each object in $\C$ has a finite $\X$-resolution.
In other words, we have $\colim_{i}\X_i=\C$, where the left-hand side is a colimit in the category of additive categories and functors.

\begin{proposition}
\label{new-lemma-3-4}
For all $i\geq 0$, the subcategory $\X_{i}$ is extension-closed in $(\C,\BE,\fs)$, and hence inherits an extriangulated structure.
\end{proposition}

\begin{proof}
We prove this by induction on $i\geq 0$. First, note that $\X_0 = \X$ is extension-closed by assumption. Thus, suppose $\X_{i}$ is extension-closed in $\C$ for some $i\geq 0$, so that we may show $\X_{i+1}$ is extension-closed. To this end, let 
$\begin{tikzcd}[column sep=0.5cm]
A \arrow{r}{f}& B\arrow{r}{g} & C \arrow[dashed]{r}{\delta}& {}
\end{tikzcd}$
be an $\fs$-triangle in $\C$ with $A,C\in\X_{i+1} = \cone(\X_{i}, \X)$. 
Then we know there are $\fs$-triangles
\begin{equation}\label{eqn:alpha}
\begin{tikzcd}[column sep=0.5cm]
Y_{A} \arrow{r}{a'}& X_{A}\arrow{r}{a} & A \arrow[dashed]{r}{\alpha}& {}
\end{tikzcd}
\end{equation}
and 
$\begin{tikzcd}[column sep=0.5cm]
Y_{C} \arrow{r}{c'}& X_{C}\arrow{r}{c} & C \arrow[dashed]{r}{\beta}& {},
\end{tikzcd}$
with $Y_{A},Y_{C}\in \X_i$ and $X_A,X_C\in \X$. 
By \cite[Prop.~3.15]{NP19}, there is a commutative diagram
\begin{equation}
\label{eqn:3-4-eqn1}
\begin{tikzcd}
{}& Y_C \arrow[equals]{r}\arrow{d}{}& Y_C \arrow{d}{c'}& {} \\ 
A \arrow{r}{f'}\arrow[equals]{d}& B' \arrow{r}{g'}\arrow{d}{h}& X_C \arrow{d}{c}\arrow[dashed]{r}{c^{*}\delta}& {} \\
A \arrow{r}{f}& B\arrow{r}{g} \arrow[dashed]{d}{g^{*}\beta}& C \arrow[dashed]{r}{\delta}\arrow[dashed]{d}{\beta}& {},\\
{} &{} &{} &{}
\end{tikzcd}    
\end{equation}
where the rows and columns are $\fs$-triangles.
In particular, we see that $h$ is an $\fs$-deflation.

As each object in $\C$ is assumed to admit a finite $\X$-resolution, there is an $\fs$-triangle of the form 
$\begin{tikzcd}[column sep=0.5cm]
Y \arrow{r}{q}& X\arrow{r}{p} & B' \arrow[dashed]{r}{\gamma}& {}
\end{tikzcd}$
with $X\in\X$ (and $Y\in \X_{r}$ for some $r\geq 0$). 
By \rm{(ET$4^{\op}$)} applied to the composition 
$\begin{tikzcd}[column sep=0.5cm]
X \arrow[two heads]{r}{p}& B' \arrow[two heads]{r}{g'}& X_C,
\end{tikzcd}$ we obtain the following commutative diagram 
\begin{equation}
\label{eqn:3-4-eqn2}
\begin{tikzcd}
Y\arrow[equals]{r}{}\arrow[]{d}{}&Y\arrow[]{d}{}&{}&{}\\
Q\arrow[]{r}{e}\arrow[]{d}{p'}&X\arrow[]{r}{g'p}\arrow[]{d}{p}&X_C\arrow[dashed]{r}{\epsilon}\arrow[equals]{d}{}&{}\\
A\arrow[]{r}{f'}\arrow[dashed]{d}{(f')^{*}\gamma}&B'\arrow[dashed]{d}{\gamma}\arrow[]{r}{g'}&X_C\arrow[dashed]{r}{c^{*}\delta}&{},\\
{}&{}&{}&{}
\end{tikzcd}
\end{equation}
where the rows and columns are $\fs$-triangles.
Since $\X$ is closed under taking cocones of $\fs$-deflations, we see that $Q = \cocone(g'p)\in\X$.

By \cite[Cor.~3.16]{NP19}, the morphism $\begin{psmallmatrix}f'a, & p\end{psmallmatrix}\colon X_A\oplus X\to B'$ is an $\fs$-deflation. Let us recall how the corresponding $\fs$-triangle is found. 
We have the $\fs$-triangle 
$\begin{tikzcd}[column sep=0.5cm]
Y \arrow{r}{q}& X\arrow{r}{p} & B' \arrow[dashed]{r}{\gamma}& {}
\end{tikzcd}$
from the middle column of \eqref{eqn:3-4-eqn2}, 
and the morphism
$f'a\colon X_{A} \to B'$. 
We realise the $\BE$-extension $(f'a)^*\gamma$ as follows:
\begin{equation}
\label{eqn:f-prime-a-gamma}
\begin{tikzcd}
Y \arrow{r}{s} \arrow[equals]{d}& D\arrow{r}{r} \arrow{d}& X_{A} \arrow{d}{f'a}\arrow[dashed]{r}{(f'a)^*\gamma}& {}\\
Y \arrow{r}{t} &X \arrow{r}{p} & B' \arrow[dashed]{r}{\gamma}& {}
\end{tikzcd}
\end{equation}
Then \cite[dual of Prop.~1.20]{LN19} yields an $\fs$-triangle 
\[
\begin{tikzcd}[column sep=1.5cm,ampersand replacement=\&]
D 
	\arrow{r}{} 
\& X_{A}\oplus X 
	\arrow{r}{\begin{psmallmatrix}f'a, & p\end{psmallmatrix}}
\& B' 
	\arrow[dashed]{r}{s_{*}\gamma}
\&{.}
\end{tikzcd}
\]

Next we show that $D\in\X_{i}$. 
Consider the $\fs$-triangles
$\begin{tikzcd}[column sep=0.5cm]
Y_{A} \arrow{r}{a'}& X_{A}\arrow{r}{a} & A \arrow[dashed]{r}{\alpha}& {}
\end{tikzcd}$ 
(see \eqref{eqn:alpha})
and
$\begin{tikzcd}[column sep=0.5cm]
Y \arrow{r}{}& Q\arrow{r}{p'} & A \arrow[dashed]{r}{(f')^{*}\gamma}& {}
\end{tikzcd}$
(from the leftmost column of \eqref{eqn:3-4-eqn2}).
Applying \cite[Prop.~3.15]{NP19}, we get a commutative diagram 
\[
\begin{tikzcd}[ampersand replacement=\&, column sep=3cm]
{} \& Y \arrow[equals]{r}{}\arrow{d}{s}\& Y\arrow{d}{} \&{}\\
Y_A \arrow{r}{}\arrow[equals]{d}{}\& D \arrow{r}{}\arrow{d}{r}\& Q\arrow[dashed]{r}{(p')^* \alpha} \arrow{d}{p'}\& {} \\
Y_A\arrow{r}{a'} \& X_A \arrow{r}{a}\arrow[dashed]{d}{a^{*}(f')^{*}\gamma=(f'a)^*\gamma}\& A \arrow[dashed]{r}{\alpha}\arrow[dashed]{d}{(f')^*\gamma}\& {} \\
{} \& {} \& {} \& {} 
\end{tikzcd}
\]
where the rows and columns are $\fs$-triangles. 
We use here that 
$
\fs((f'a)^{*}\gamma) = 
[\begin{tikzcd}[column sep=0.5cm] Y \arrow{r}{s}& D\arrow{r}{r} & X_{A}\end{tikzcd}]
$
as in \eqref{eqn:f-prime-a-gamma}.
Since $Y_A\in\X_{i}$ by assumption and $Q\in\X\sse\X_{i}$, we see that $D\in\X_{i}$ because $\X_i$ is extension-closed by the inductive hypothesis.

Now we are in position to show $B\in\X_{i+1}$ as needed. 
Applying axiom \rm{(ET$4^{\op}$)} to the composition
$\begin{tikzcd}[column sep=1.2cm, ampersand replacement=\&]
    X_A \oplus X \arrow[two heads]{r}{\begin{psmallmatrix}
        f'a, & p
    \end{psmallmatrix}}\& B' \arrow[two heads]{r}{h}\& B
\end{tikzcd}$ (recall $h$ was obtained in \eqref{eqn:3-4-eqn1}) yields a commutative diagram
\begin{equation*}
\begin{tikzcd}[column sep=2cm,ampersand replacement=\&]
D \arrow[equals]{r}{}\arrow{d}{}\& D \arrow{d}
\& {} \& {}\\
Y_B \arrow{r}{}\arrow{d}{}\& X_A\oplus X \arrow{r}{h\circ\begin{psmallmatrix}
    f'a, & p
\end{psmallmatrix}}\arrow{d}{\begin{psmallmatrix}
    f'a, & p
\end{psmallmatrix}}\& B \arrow[dashed]{r}{}\arrow[equals]{d}{}\& {}\\
Y_C \arrow{r}{}\arrow[dashed]{d}{}\& B' \arrow{r}{h}\arrow[dashed]{d}{s_{*}\gamma}\& B\arrow[dashed]{r}{g^*\beta} \& {}\\
{}\&{}\&{}\&{}
\end{tikzcd}
\end{equation*}
where the rows and columns are $\fs$-triangles. 
Note that $D,Y_C\in\X_i$ implies $Y_B\in\X_i$ by extension-closure. 
As $X_A\oplus X\in\X$, we see that 
$B=\cone(\begin{tikzcd}
    Y_B \arrow[tail]{r}& X_A \oplus X
\end{tikzcd})\in\cone(\X_i,\X) = \X_{i+1}$
and we are done. 
\end{proof}

We denote the inherited extriangulated structure on $\X_i$ by 
$(\X_{i},\BE|_{\X_{i}},\fs|_{\X_{i}})$ (see \S\ref{sec:ext-closed-extri-cats}). 
Now we investigate the Grothendieck groups $K_0(\X_{i})\deff K_0(\X_{i},\BE|_{\X_{i}},\fs|_{\X_{i}})$ of these extriangulated subcategories of $(\C,\BE,\fs)$.

\begin{proposition}
\label{prop_sequence_of_K0}
The ascending chain $\X_0\subseteq \X_1\subseteq \X_2\subseteq \cdots$ induces a sequence of isomorphisms $K_0(\X_0)\cong K_0(\X_1)\cong K_0(\X_2)\cong \cdots$ of Grothendieck groups.
\end{proposition}
\begin{proof}
Fix an integer $i\geq 0$.
We shall show that the natural group homomorphism
$\phi\colon K_0(\X_i)\to K_0(\X_{i+1})$
given by $\phi([C]) = [C]$ 
is an isomorphism by constructing an inverse as follows.
Let $C\in\X_{i+1} = \cone(\X_{i},\X)$ be arbitrary. Then there is an $\fs$-conflation 
$P_1\overset{p'}{\lra} P_0\overset{p}{\lra} C$ 
with $P_0\in\X\sse\X_{i}$ and $P_1\in\X_i$. 
We claim that the assignment 
$\psi\colon [C]\mapsto [P_0]-[P_1]$ 
gives rise to a group homomorphism 
$\psi\colon K_0(\X_{i+1})\to K_0(\X_{i})$.

To see this, let $Q_1\overset{q'}{\lra} Q_0\overset{q}{\lra} C$ be any $\mathfrak{s}$-conflation with $Q_0, Q_1\in\X_i$. 
By \cite[Prop.~3.15]{NP19}, there is a commutative diagram
\[
\begin{tikzcd}
{}&Q_1\arrow[equals]{r}{}\arrow{d}{}&Q_1\arrow{d}{}\\
P_1\arrow{r}{}\arrow[equals]{d}{}&P\arrow{r}{}\arrow{d}{}&Q_0\arrow{d}{}\\
P_1\arrow{r}{}&P_0\arrow{r}{}&C
\end{tikzcd}
\]
where all rows and columns are $\mathfrak{s}$-conflations.
Since $\X_i$ is extension-closed by \cref{new-lemma-3-4}, we see that $P$ lies in $\X_i$. Hence, in $K_0(\X_i)$ we have 
$[P_1]+[Q_0] 
= [P]
= [Q_1]+[P_0]$,
whence
$[Q_0]-[Q_1] = [P_0]-[P_1]$ and 
$\psi\colon \Ob(\X_{i+1}) \to K_{0}(\X_{i})$ is well-defined.
It is straightforward to check that $\psi$ induces a group homomorphism as claimed, and that $\phi$ and $\psi$ are mutually inverse. 
\end{proof}

We are now in position to prove \cref{thm_Quillen_resolution}.

\begin{proof}[Proof of \cref{thm_Quillen_resolution}]
We will show that $K_0(\C) = K_0(\C,\BE,\fs)$ satisfies the universal property of the filtered colimit $\colim_{i} K_0(\X_i)$. 
First, for each integer $i\geq 0$, there are the canonical group homomorphisms 
$\alpha_i \colon K_0(\X_i)\to K_0(\C)$
and 
$\beta_i \colon K_0(\X_i)\to \colim_{i} K_0(\X_i)$
given by $[C]\mapsto [C]$. 
It follows that there is a unique group homomorphism $\gamma\colon\colim_{i} K_0(\X_i) \to K_0(\C)$ given by $[C]\mapsto [C]$.

On the other hand, let $C\in\C$ be arbitrary and suppose \eqref{seq_X_resolution} is an $\X$-resolution of $C$. 
Define $\delta\colon \Ob(\C)\to \colim_{i} K_0(\X_i)$ by 
$\delta(C) = \sum_{i=0}^n (-1)^{i} [X_i]$. To see that this is independent of the chosen $\X$-resolution, suppose 
\begin{equation}
\label{eqn:second-filt}
\begin{tikzcd}[column sep=0.5cm]
    Y_m \arrow{rr}\arrow[equals]{dr}&& Y_{m-1}\arrow{rr}\arrow{dr} && Y_{m-2}\arrow{r} &\cdots\arrow{r}& Y_1 \arrow{rr}\arrow{dr}&& Y_0 \arrow{dr}\arrow{rr}&& C\\
    & D_m \arrow{ur}&& D_{m-1}\arrow{ur}&&&&D_1 \arrow{ur}&& D_0\arrow[equals]{ur}
\end{tikzcd}
\end{equation}
is another $\X$-resolution of $C$. 
Set $N\deff \max\{n,m\}$. Then all the $\fs$-conflations involved in \eqref{seq_X_resolution} and \eqref{eqn:second-filt} are $\fs|_{\X_{N}}$-conflations. In particular, in $K_0(\X_{N})$ and hence also in $\colim_{i} K_0(\X_i)$, we have that 
$
\sum_{i=0}^n (-1)^{i} [X_i]
    = [C]
    = \sum_{i=0}^m (-1)^{i} [Y_i]
$.
It is clear that $\delta$ is constant on isoclasses of objects. 

Furthermore, any $\mathfrak{s}$-conflation $A\lra B\lra C$ in $(\C,\BE,\fs)$ lies in $\X_j$ for some sufficiently large integer $j$. 
Thus, by definition, $[A]-[B]+[C]=0$ occurs in $K_0(\X_j)$ and also in $\colim_{i} K_0(\X_i)$. 
Hence, $\delta$ induces a group homomorphism 
$\delta\colon K_0(\C)\to \colim_{i} K_0(\X_i)$. 
Finally, it is clear that $\gamma$ and $\delta$ are mutually inverse.
As the canonical morphism  $K_0(\X_{i})\xto{\cong}K_0(\X_{i+1})$ is an isomorphism for any $i\geq 0$ by \cref{prop_sequence_of_K0}, the canonical morphism $K_0(\X)\to \colim_{i} K_0(\X_i)$ is also an isomorphism and we are done.
\end{proof}

\section{Applications to the index in triangulated categories}
\label{sec_index}

As an application of our Extriangulated Resolution Theorem (\cref{thm_Quillen_resolution}), in \S\ref{subsec:Palu-JS-indices} we recover some index isomorphisms that have recently appeared in the literature. 
This allows us to propose an index with respect to an extension-closed subcategory in \S\ref{subsec_additive_formula}, as well as establish an additivity formula for this index. 
Lastly, we prove a generalization of  Fedele's \cite[Thm.~C]{Fed21} in \S\ref{subsec_d_angulated_cat}.


\subsection{The PPPP and JS index isomorphisms}
\label{subsec:Palu-JS-indices}
Throughout \S\ref{subsec:Palu-JS-indices}, we fix the following setup.

\begin{setup}
\label{setup:Palu-JS-indices}
Suppose $(\C,\mathbb{E},\mathfrak{s})$  corresponds to a skeletally small, idempotent complete, triangulated category $\C$.
We also assume $\X\subseteq\C$ is an additive subcategory that is contravariantly finite, rigid and closed under direct summands.
\end{setup}

Recall from \cref{example:n-CT-resolution} that if $\X$ is $n$-cluster tilting ($n\geq 2$), then 
any object $C\in\C$ admits a finite $\X$-resolution
$
X_{n-1}\to \cdots\to X_1\to X_0\to C
$ 
in $\C^\X_R$ 
of length (at most) $n-1$.

\begin{definition}\label{def_index}
\cite[Def.~2.1]{Pal08} 
Suppose $\X$ is a $2$-cluster tilting subcategory of $\C$. 
The \emph{Palu index with respect to $\X$} of $C$ is the element 
$\ind_\X(C)\deff [X_0]^{\sp}-[X_1]^{\sp}\in K^{\mathsf{sp}}_0(\X)$, 
where $X_1\to X_0\to C$ is an $\X$-resolution in $\C^\X_R$ of $C$.
\end{definition}

It is straightforward to check that the following isomorphism is a direct consequence of \cite[Prop.~4.11]{PPPP19}. 
We call it the \emph{PPPP index isomorphism}; it is a special case of \cref{cor_JS_index}.

\begin{corollary}[PPPP index isomorphism]
\label{cor_Palu_index}
Suppose $\X$ is $2$-cluster tilting in $\C$. 
Then the Palu index $\ind_\X$ yields the following isomorphism of abelian groups.
\begin{align}
K_0(\C^\X_R) 	&\overset{\cong}{\longleftrightarrow}K_0^{\sp}(\X) \label{eq_PPPP}\\
[C]^\X_R &\longmapsto \ind_\X(C)  \nonumber\\
[X]^\X_R &\longmapsfrom [X]^{\sp}\nonumber
\end{align}
\end{corollary}

The PPPP index isomorphism shows that the class $[C]^\X_R$ in the Grothendieck group $K_0(\C^\X_R)$ of an object $C\in\C$ can be interpreted as the Palu index of $C$ with respect to $\X$. This suggests that one can define an index using the relative extriangulated structure $\C^\X_R$, even without the $2$-cluster tilting assumption. 
This is what is done in \cite{JS23}.

\begin{definition}\label{def:JS-index}
\cite[Def.~3.5]{JS23} 
The 
\emph{J\o{}rgensen--Shah index of $C\in\C$ with respect to $\X$} is $\ind_\X(C)\deff[C]_R^\X\in K_0(\C_R^\X)$.
\end{definition}

An analogue of \cref{cor_Palu_index} for an $n$-cluster tilting subcategory $\X$ ($n\geq 2$) was given in \cite{JS23}. 
We call the following isomorphism the \emph{JS index isomorphism} since it generalizes \cite[Thm.~4.10]{JS23} by removing several restrictions on the ambient triangulated category $\C$. 
As we will see, it is a special case of \cref{cor_Quillen_resolution_to_index}.

\begin{corollary}[JS index isomorphism]
\label{cor_JS_index}
Suppose $\X$ is $n$-cluster tilting in $\C$. 
Then there is an isomorphism of abelian groups
\begin{align}
K_{0}(\C_R^\X) & \overset{\cong}{\longleftrightarrow}K_0^{\mathsf{sp}}(\X) \label{eq_JS_index}\\
[C]^\X_R & \longmapsto \sum_{i=0}^{n-1}(-1)^i[X_i]^{\mathsf{sp}}  \nonumber\\
[X]^\X_R & \longmapsfrom [X]^{\sp}, \nonumber
\end{align}
where 
$C\in\C$ admits an $\X$-resolution 
$
X_{n-1}\to \cdots\to X_1\to X_0\to C
$
in $\C^\X_R$ of length $n-1$.
\end{corollary}

Using \cref{lem_comparison_relative_structures} to see 
$\C^R_\N = \C_R^\X$ when $\N=\X^{\perp_0}$, it is then clear that isomorphisms \eqref{eq_PPPP} and \eqref{eq_JS_index} above are special cases of our next result.

\begin{corollary}
\label{cor_Quillen_resolution_to_index}
Put $\N\deff\X^{\perp_0}$. 
If each object $C\in\C^R_\N$ admits a finite $\X$-resolution in $\C^R_\N$, then 
there is 
an abelian group isomorphism
\begin{align}
K_{0}(\C^R_\N) & \overset{\cong}{\longleftrightarrow}K_0^{\mathsf{sp}}(\X) \label{eq_generalized_index1}\\
[C]_\N^R & \longmapsto \sum_{i=0}^{m}(-1)^i[X_i]^{\mathsf{sp}}  \nonumber\\
[X]_\N^R & \longmapsfrom [X]^{\sp}, \nonumber
\end{align} 
where 
$C$ admits an $\X$-resolution 
$
X_{m}\to \cdots\to X_1\to X_0\to C
$
in $\C^R_\N$ of length $m$.
\end{corollary}

\begin{proof}
Since $\X$ is closed under direct summands, one can check that $\X$ coincides with the subcategory of $\BE^R_\N$-projectives in $\C^R_\N$. 
In particular, it is extension-closed and closed under taking cocones of $\mathfrak{s}_\N^R$-deflations in $\C^R_\N$.
Furthermore, the inherited extriangulated structure $(\X,\BE^R_\N|_{\X},\fs^R_\N|_{\X})$ is just the split exact structure on $\X$. Thus, 
$K_0(\X,\BE^R_\N|_{\X},\fs^R_\N|_{\X}) = K_0^{\sp}(\X)$. 
Hence, the assertion follows directly from \cref{thm_Quillen_resolution}.
\end{proof}

\subsection{Indices with respect to an extension-closed subcategory of a triangulated category}
\label{subsec_additive_formula}

We begin by proposing two indices. 

\begin{definition}\label{def:index-wrt-extn-closed}
Let $\N$ be an extension-closed subcategory of a skeletally small, idempotent complete, triangulated category $(\C,\BE,\fs)$. 
We call the class 
$[C]^R_\N\in K_{0}(\C^R_\N)$ 
(resp.\ $[C]^L_\N\in K_{0}(\C^L_\N)$) 
the \emph{right index} 
(resp.\ the \emph{left index}) \emph{with respect to $\N$} of $C\in\C$. 
\end{definition}

We make some comments on how our work in this subsection relates to \cite{JS23}.

\begin{remark}\label{rem:connection-to-JS}
Suppose $(\C,\mathbb{E},\mathfrak{s})$ is a skeletally small, idempotent complete, triangulated category with suspension $[1]$. In addition, 
let $\X$ be an additive, contravariantly finite, rigid subcategory $\X\sse\C$ that is closed under isomorphisms and direct summands. 
Putting $\N\deff\X^{\perp_0}$, we get 
$\C^R_\N = \C_R^\X$
by \cref{lem_comparison_relative_structures}. 
In particular, we see that $[C]_\N^R = [C]_R^\X = \ind_\X(C)$ is the J{\o}rgensen--Shah index of $C\in\C$. 
Thus, we can think of $[-]_\N^R$ as a generalization of J{\o}rgensen--Shah's index.
In the same way, the assignment $[-]^L_\N$ can also be regarded as an index. 
This motivates \cref{def:index-wrt-extn-closed}.

In \cite[Prop.~3.12]{JS23}, 
a homomorphism $\theta_{\X}\colon K_0(\mod\X)\to K_0(\C^\X_R)$
was produced that measures how far the index $\ind_\X$ is from being additive on triangles (see \cite[Thm.~3.14]{JS23}). 
We recall that $\theta_{\X}$ is given by 
$\theta_{\X}(\C(-,C)|_\X) = [C]^\X_R + [C[-1]]^\X_R$ for $C\in\C$.
We will define an analogue of $\theta_\X$ of the form $\theta^R_\N\colon K_0(\C/\N)\to K_0(\C^R_\N)$, which is more natural in our framework (see \cref{prop_add_formula3}). However, in the setup of \cite[\S3]{JS23}, one can pass between $\theta_\X$ and $\theta^R_\N$ using the equivalence $G\colon \C/\N \overset{\simeq}{\lra} \mod\X$ (see \cref{ex_BM_BBD}\ref{example:rigid}). 
Indeed, if 
\begin{equation}\label{eqn:G-star}
G_*\colon K_0(\C/\N) \overset{\cong}{\lra} K_0(\mod\X)
\end{equation}
is the induced isomorphism, then 
$\theta_\X G_* = \theta_\N^R$.

Just like in \cite{JS23}, we determine additivity formulae for our left and right indices in this subsection; see \cref{thm_JS_additivity}. 
Although the approach we take is similar to \cite[\S3]{JS23}, 
some arguments are simplified by utilizing the localization theory of extriangulated categories.
Furthermore, we also show that the image of $\theta^R_\N$ captures the kernel of the natural surjection 
$K_0(\C^R_\N)\onto K_0(\C)$ (see \cref{prop_comparison}).
\end{remark}

Note that our right index is a strict generalization of the J{\o}rgensen--Shah index since there are extension-closed subcategories of triangulated categories that do not arise as $\X^{\perp_{0}}$ for any contravariantly finite, rigid subcategory $\X$.

\begin{example}\label{example:extension-closed-not-rigid-perp}
Consider the quiver $1 \to 2$, its path algebra $\Lambda$ over a field $k$ and the bounded derived category $\C$ of finite-dimensional $\Lambda$-modules. The module category $\N\deff\mod\Lambda$ is extension-closed in $\C$, but a straightforward argument shows it cannot be of the form $\X^{\perp_{0}}$ for any rigid subcategory $\X\sse\C$.
\end{example}

\begin{setup}
\label{setup:additivity-formula}
In the remainder of \S\ref{subsec_additive_formula}, we suppose $(\C,\mathbb{E},\mathfrak{s})$  corresponds to a skeletally small, idempotent complete, triangulated category $\C$ with suspension $[1]$.
We also assume $\N$ is an extension-closed subcategory of $\C$, such that $\cone(\N,\N)=\C$. 
\end{setup}

We remind the reader that under \cref{setup:additivity-formula} all the results from \S\ref{subsub_localization} apply.
In particular, $\ol{\Sn}$ is a multiplicative system in $\overline{\C}\deff \C/[\N]$ (this is condition (MR$2$) from \cite{NOS22} and follows from \cref{thm_abel_loc1}\ref{N-thick}), 
the localization $\C/\N = \ol{\C}[\ol{\Sn}^{-1}] \cong \C[\Sn^{-1}]$ is an abelian category and  
the localization functor $Q\colon \C \to \C/\N$ 
is cohomological (see \cref{lem:GZ-localization-isomorphisms} and \cref{N-serre-local-is-abelian}), and $Q$ induces a right exact functor $Q\colon \C^R_\N \to \C/\N$ (see \cref{cor_abel_loc2}). 
In the sequel we focus on the right index and state, but do not prove, the corresponding assertions for the left index.

Our goal is to establish \cref{mainthm:additivity}. In order to define $\theta^R_\N$, we need two preliminary results. 
Note that the assumption $\cone(\N,\N)=\C$ is not needed to prove \cref{lem_add_formula1}.

\begin{lemma}
\label{lem_add_formula1}
(cf.\ \cite[Lem.~3.8]{JS23}) 
Let $B,C\in\C$ be objects with an isomorphism $Q(B)\cong Q(C)$ in $\C/\N$.
Then we have equalities
\[
[B]^R_\N+[B[-1]]^R_\N=[C]^R_\N+[C[-1]]^R_\N
\hspace{1cm}\text{and}\hspace{1cm}
[B]^L_\N+[B[1]]^L_\N=[C]^L_\N+[C[1]]^L_\N
\]
in $K_{0}(\C^R_\N)$ and $K_{0}(\C^L_\N)$, respectively.
\end{lemma}

\begin{proof}
We show the first equation.
Since $\ol{\Sn}$ is a multiplicative system in $\ol{\C}$ (i.e.\ admits a calculus of left and right fractions in $\ol{\C}$), an isomorphism 
$\alpha\colon Q(B)\overset{\cong}{\lra}Q(C)$ 
is represented by a roof diagram 
$\begin{tikzcd}[column sep=0.5cm]
    B \arrow{r}{\overline{s}}& C' & C\arrow{l}[swap]{\overline{t}}
\end{tikzcd}$
in $\overline{\C}$ with $\ol{t}\in\ol{\Sn}$, i.e.\ $t\in\Sn$. 
Since $\alpha$ and $Q(t)$ are isomorphisms, we have that $Q(s) = Q(t)\alpha$ is an isomorphism. 
As the class $\Sn$ is saturated, we see $s\in\Sn$ too. 
So it is enough to show $[B]^R_\N+[B[-1]]^R_\N=[C']^R_\N+[C'[-1]]^R_\N$.
By \cite[Lem.~2.6]{Oga22b}, 
the morphism $s$ is part of a triangle
\begin{equation}\label{eqn:lemma-5.10-triangle}
A\overset{f}{\lra} B \overset{s}{\lra} C'\overset{h}{\lra} A[1] 
\end{equation}
with $f$ and $h$ factoring through objects in $\N$.
Since 
$A\overset{f}{\lra} B \overset{s}{\lra} C'$ is an $\mathfrak{s}^R_\N$-conflation, 
we get 
\begin{equation}\label{eqn:lemma-5.10-eq1}
    [A]^R_\N-[B]^R_\N+[C']^R_\N=0
\end{equation}
in $K_{0}(\C^R_\N)$.
In addition, we obtain the  $\mathfrak{s}^R_\N$-conflation 
$\begin{tikzcd}[column sep=1.4cm]
    B[-1] \arrow{r}{-s[-1]}& C'[-1] \arrow{r}{-h[-1]}& A 
\end{tikzcd}$
by rotating the triangle \eqref{eqn:lemma-5.10-triangle} twice, 
and hence 
\begin{equation}\label{eqn:lemma-5.10-eq2}
    [B[-1]]^R_\N-[C'[-1]]^R_\N+[A]^R_\N=0.
\end{equation}
The desired equality follows from combining \eqref{eqn:lemma-5.10-eq1} and \eqref{eqn:lemma-5.10-eq2}.
\end{proof}

\begin{lemma}\label{lem_add_formula2}
(cf.\ \cite[Lem.~3.11]{JS23}) 
Let 
$\begin{tikzcd}[column sep=0.5cm]
    0 \arrow{r}& A \arrow{r}{\alpha}& B \arrow{r}{\beta}& C \arrow{r}& 0
\end{tikzcd}$
be a short exact sequence in $\C/\N$.
Then we have equalities
\begin{eqnarray*}
\left( [A]^R_\N+[A[-1]]^R_\N \right) + \left( [C]^R_\N+[C[-1]]^R_\N \right) &=& [B]^R_\N+[B[-1]]^R_\N, \\
\left( [A]^L_\N+[A[1]]^L_\N \right) + \left( [C]^L_\N+[C[1]]^L_\N \right) &=& [B]^L_\N+[B[1]]^L_\N
\end{eqnarray*}
in $K_{0}(\C^R_\N)$ and $K_{0}(\C^L_\N)$, respectively.
\end{lemma}

\begin{proof}
We show the first equation.
By \cite[Lem.~3.32]{NOS22}, 
there exists an $\mathfrak{s}_\N$-triangle 
\begin{equation}\label{eqn:lemma-5.11-Sn-triangle}
\begin{tikzcd}
    A' \arrow{r}{f'}& B' \arrow{r}{g'}& C' \arrow[dashed]{r}{h'}&{}
\end{tikzcd}
\end{equation}
and an isomorphism 
\begin{equation*}\label{eqn:lemma-5.11-ses}
\begin{tikzcd}
    0 \arrow{r}{}& QA' \arrow{r}{Qf'}\arrow{d}{\cong}& QB' \arrow{r}{Qg'}\arrow{d}{\cong}& QC' \arrow{r}{}\arrow{d}{\cong}& 0 \\
    0 \arrow{r}{}& QA \arrow{r}{\alpha}& QB \arrow{r}{\beta}& QC \arrow{r}{}& 0
\end{tikzcd}
\end{equation*}
of short exact sequences
in $\C/\N$.
Since \eqref{eqn:lemma-5.11-Sn-triangle} is an $\fs_\N$-triangle, we know that 
$h'[-1]$ and $h'$ 
both factor through $\N$ by \cref{lem_specific_relative_structures}. 
Thus, \eqref{eqn:lemma-5.11-Sn-triangle} and 
$\begin{tikzcd}[cramped, column sep=1.3cm]
    A'[-1] \arrow{r}{}&[-0.8cm] B'[-1] \arrow{r}{}&[-0.8cm] C'[-1] \arrow[dashed]{r}{-h'[-1]}&{}
\end{tikzcd}$
are both $\mathfrak{s}^R_\N$-triangles.
So 
$
[A']_\N^R-[B']_\N^R+[C']_\N^R=0=-([A'[-1]]_\N^R-[B'[-1]]_\N^R+[C'[-1]]_\N^R),
$
and combining with \cref{lem_add_formula1} yields the desired equation.
\end{proof}

We are now in a position to define $\theta^R_\N$. We remark that instances of this homomorphism have appeared in \cite[\S2.1]{Pal08}, \cite[\S4]{Jor21} and \cite[\S2]{Fed21} before.

\begin{proposition}
\label{prop_add_formula3}
(cf.\ \cite[Prop.~3.12]{JS23}) 
The following are well-defined group homomorphisms.
\begin{eqnarray*}
&&\theta^R_\N\colon K_0(\C/\N)\lra K_{0}(\C^R_\N) 
\hspace{0.5cm}\text{ given by }\hspace{0.5cm}
    [Q(C)]\longmapsto [C]_\N^R+[C[-1]]_\N^R\\
&&\theta^L_\N\colon K_0(\C/\N)\lra K_{0}(\C_\N^L) 
\hspace{0.5cm}\text{ given by }\hspace{0.5cm}
    [Q(C)]\longmapsto [C]_\N^L+[C[1]]_\N^L
\end{eqnarray*}
\end{proposition}

\begin{proof}
We only check the well-definedness of $\theta^R_\N$. 
Let $\isoclass(\C/\N)$ denote the set of isoclasses of objects in $\C/\N$. 
By \cref{lem_add_formula1}, we get a well-defined map
$\isoclass(\C/\N) \to K_{0}(\C^R_\N)$ given by 
$[Q(C)] \mapsto [C]_\N^R+[C[-1]]_\N^R$.
Due to \cref{lem_add_formula2}, this map is additive on short exact sequences in $\C/\N$, 
inducing a well-defined homomorphism 
$\theta^R_\N\colon K_0(\C/\N) \to K_{0}(\C^R_\N)$ as claimed.
\end{proof}

The additivity formulae with error terms for the left and right indices can now be established. 

\begin{theorem}
\label{thm_JS_additivity}
(cf.\ \cite[Thm.~3.14]{JS23}) 
Let $A\overset{f}{\lra}B\overset{g}{\lra}C\overset{h}{\lra} A[1]$ be any triangle in $(\C,\mathbb{E},\mathfrak{s})$.
Then we have equalities
\[
[A]_\N^R-[B]_\N^R+[C]_\N^R=\theta^R_\N\bigl(\Im Q(h)\bigr)
\hspace{1cm}\text{and}\hspace{1cm}
[A]_\N^L-[B]_\N^L+[C]_\N^L=\theta^L_\N\bigl(\Im Q(h[-1])\bigr)
\]
in $K_{0}(\C^R_\N)$ and $K_{0}(\C^L_\N)$, respectively.
\end{theorem}
\begin{proof}
Since $\cone(\N,\N)=\C$, there is a triangle 
$N_1\overset{a}{\lra} N_{0}\overset{b}{\lra} B[1]\overset{c}{\lra} N_1[1]$ in $\C$ with $N_i\in\N$.
Taking a homotopy pullback of $-f[1]$ along $b$ by the octahedral axiom, we have the commutative diagram
\begin{equation}\label{eqn:thm-5.13-big-diagram}
\begin{tikzcd}
    &&N_1\arrow[equals]{r}{}\arrow{d}{}&N_1\arrow{d}{a}\\
    N_{0}[-1]\arrow{r}{}\arrow{d}{}&C\arrow{r}{h_1}\arrow[equals]{d}{}&X\arrow{r}{d}\arrow{d}[swap]{h_2}\wPB{dr}&N_{0}\arrow{d}{b}\\
    B\arrow{r}{g}&C\arrow{r}{h}&A[1]\arrow{r}[swap]{-f[1]}\arrow{d}{}&B[1]\arrow{d}{c}\\
    &&N_1[1]\arrow[equals]{r}{}&N_1[1]
\end{tikzcd}
\end{equation}
of triangles in $\C$. 
From \eqref{eqn:thm-5.13-big-diagram}, and by rotating if necessary, we can extract the three $\mathfrak{s}^R_\N$-triangles:
\begin{align*}
    \begin{tikzcd}[ampersand replacement=\&, column sep=1.3cm]N_{0}[-1] \arrow{r}\& C \arrow{r}\& X \arrow[dashed]{r}{d}\&{,}\end{tikzcd}\\
    \begin{tikzcd}[ampersand replacement=\&, column sep=1.3cm]N_{1}[-1] \arrow{r}\& X[-1] \arrow{r}\& A \arrow[dashed]{r}{c[-1]f}\&{,}\end{tikzcd}\\
    \begin{tikzcd}[ampersand replacement=\&, column sep=1.3cm]N_{1}[-1] \arrow{r}\& N_{0}[-1] \arrow{r}\& B \arrow[dashed]{r}{-c[-1]}\&{.}\end{tikzcd}
\end{align*}
These imply the equality $[A]_\N^R-[B]_\N^R+[C]_\N^R=[X]_\N^R+[X[-1]]_\N^R$ in $K_{0}(\C^R_\N)$. 
As $h = h_2 h_1$, 
it follows from \cite[Lem.~2.6]{Oga22b} that $Q(h)$ 
is the epimorphism $Q(h_1)$ followed by the monomorphism $Q(h_2)$.
Thus, we have $Q(X)\cong \Im Q(h_1)\cong \Im Q(h)$.
\cref{prop_add_formula3} ensures $\theta^R_\N(\Im Q(h))=[X]_\N^R+[X[-1]]_\N^R$ and we are done.
\end{proof}

In our next result, which implies \cref{mainprop:E}, we produce right exact sequences connecting the Grothendieck groups of $\C/\N$, $\C^R_\N$ and $\C$ using 
the canonical surjections
\begin{eqnarray*}
&&\pi^R_{\N}\colon K_{0}(\C^R_\N)\onto K_0(\C)\hspace{0.5cm}\text{ given by }\hspace{0.5cm}[C]^R_\N\mapsto [C],\\
&&\pi^L_{\N}\colon K_{0}(\C^L_\N)\onto K_0(\C)\hspace{0.5cm}\text{ given by }\hspace{0.5cm}[C]^L_\N\mapsto [C].
\end{eqnarray*}

\begin{proposition}
\label{prop_comparison}
We have the following right exact sequences of Grothendieck groups.
\begin{eqnarray}
&&  \begin{tikzcd}[ampersand replacement=\&]
        K_0(\C/\N)
            \arrow{r}{\theta^R_\N}
        \& K_{0}(\C^R_\N)
            \arrow{r}{\pi^R_{\N}}
        \& K_0(\C)
            \arrow{r} 
        \& 0
    \end{tikzcd} \label{eq_comparison2}\\
&&  \begin{tikzcd}[ampersand replacement=\&]
        K_0(\C/\N)
            \arrow{r}{\theta_\N^L}
        \& K_{0}(\C^L_\N)
            \arrow{r}{\pi^L_{\N}}
        \& K_0(\C)
            \arrow{r} 
        \& 0
    \end{tikzcd} \label{eq_comparison1}
\end{eqnarray}
\end{proposition}

\begin{proof}
We prove \eqref{eq_comparison2} is right exact. 
Since $\pi^R_{\N}$ is surjective, it suffices to show that $\Ker\pi^R_{\N} = \Im\theta^R_\N$. 
\cref{thm_JS_additivity} tells us that 
\[
\Ker \pi^R_{\N} 
    = \Braket{[A]_\N^R-[B]_\N^R+[C]_\N^R | A\lra B\lra C\overset{h}{\lra} A[1] \text{ is a triangle in }\C}
    \sse \Im\theta^R_\N.
\]
For the other containment, take any element $[C]^R_\N+[C[-1]]^R_\N\in\Im\theta^R_\N$.
The existence of the triangle 
$C[-1]\lra 0\lra C\overset{\id_{C}}{\lra} C$ implies 
$
\pi^R_{\N}([C]^R_\N+[C[-1]]^R_\N) 
    = [C]+[C[-1]] 
    = [0]
$ 
in $K_0(\C)$. 
Hence, $\Ker\pi^R_{\N}= \Im\theta^R_\N$ and we are done.
\end{proof}

The following non-commutative diagram summarizes the main abelian group homomorphisms that have appeared so far (cf.\ \eqref{diag_one_sided_localization}). 
By \cref{thm_abel_loc1}\ref{extri-local}, the localization functor $(Q,\mu)\colon \C_\N\to (\C/\N,\widetilde{\BE_\N},\widetilde{\fs_\N})$ is exact. Hence, it induces a group homomorphism 
$p\colon K_0(\C_\N)\to K_0(\C/\N)$. 
Furthermore, it is surjective by Enomoto--Saito \cite[Cor.~4.32]{ES22}. 
\begin{equation}
\label{diag_bridges_of_K0}
\begin{tikzcd}[column sep = 1.5cm,row sep=0.2cm]
    &&\Ker \pi^R_{\N} = \Im\theta_{\N}^{R}\arrow[hook]{dl}&\\
    &K_{0}(\C^R_\N)\arrow[two heads]{dl}[swap]{\pi^R_{\N}}\commutes[\circlearrowleft]{dd}&&\\
    K_0(\C)&&K_{0}(\C_\N)\arrow[two heads]{ul}\arrow[two heads]{dl}\arrow[two heads]{r}{p}
        \commutes[\not\circlearrowleft]{uu}
        \commutes[\not\circlearrowleft]{dd}
    &K_0(\C/\N)\arrow[two heads, bend left]{ddl}{\theta_{\N}^{L}}\arrow[two heads, bend right]{uul}[swap]{\theta_{\N}^{R}}\\
    &K_{0}(\C_\N^L)\arrow[two heads]{ul}{\pi^L_{\N}}&&\\
    &&\Ker\pi^L_{\N} = \Im\theta_{\N}^{L}\arrow[hook]{ul}&
\end{tikzcd}
\end{equation}

\subsection{Connection to higher homological algebra}
\label{subsec_d_angulated_cat}

In \cite{Fed21}, Fedele exhibited the connection between the Grothendieck group of a triangulated category and that of an $n$-cluster tilting subcategory via J{\o}rgensen's triangulated index \cite[Def.~3.3]{Jor21}. In this subsection, using \cref{prop_comparison}, we generalize two key results from \cite{Fed21}, namely \cite[Prop.~3.5 and Thm.~C]{Fed21}.
Our setup here is as follows, which is a special case of \cref{setup:additivity-formula} due to  \cref{lem_comparison_relative_structures}.

\begin{setup}\label{setup:Fedele}
In this subsection, we suppose $(\C,\mathbb{E},\mathfrak{s})$  corresponds to a skeletally small, idempotent complete, triangulated category $\C$ with suspension $[1]$.
We also fix an $n$-cluster tilting subcategory $\X\subseteq \C$ and put $\N\deff\X^{\perp_0}$.
\end{setup}

Palu has previously shown that an exact sequence of the form \eqref{eq_comparison2} arises from a $2$-cluster tilting subcategory of an algebraic triangulated category (see \cite[Lem.~9]{Pal09}). 
We prove an exact sequence of the same form exists for $\X$ in \cref{cor:palu-lemma-9-nCT} below, but with fewer restrictions on $\C$. 
We also note that Palu used homotopy categories and t-structures to prove \cite[Lem.~9]{Pal09}, whereas we use different methods.
We denote the JS index isomorphism (see \cref{cor_JS_index}) by  
\begin{equation}\label{eqn:triangulated-index}
\index_\X 
    \colon K_0(\C_R^\X) = K_0(\C^R_\N) 
        \overset{\cong}{\lra} K_0^{\sp}(\X).
\end{equation}
Recall that $G_{*}$ first appeared in \cref{rem:connection-to-JS}.

\begin{corollary}
\label{cor:palu-lemma-9-nCT}
There exists an exact sequence of Grothendieck groups as follows. 
\begin{equation}\label{eqn:nCT-right-exact-seq}
\begin{tikzcd}[column sep=2.5cm]
    K_0(\mod\X) 
        \arrow{r}{\index_\X \theta^R_\N G_*^{-1}}
    & K_0^{\mathsf{sp}}(\X)
        \arrow{r}{\pi^R_{\N} \index_\X^{-1}}
    & K_0(\C)
        \arrow{r}{}
    & 0
\end{tikzcd}
\end{equation}
\end{corollary}
\begin{proof}
Applying \cref{prop_comparison}, we obtain the exact sequence \eqref{eq_comparison2}. 
It follows that \eqref{eqn:nCT-right-exact-seq} is right exact by twisting \eqref{eq_comparison2} with the isomorphisms \eqref{eqn:G-star} and 
\eqref{eqn:triangulated-index}.
\end{proof}

\begin{remark}\label{rem:improve-Fedele}
It is established in \cite[Lem.~3.4]{Fed21} that the assignment 
$
[C] 
    \mapsto \index_\X([C]_R^\X) 
    = \sum_{i=0}^{n-1}(-1)^i[X_i]^{\mathsf{sp}}
    \in K_0^{\sp}(\X)
$
for $[C]\in K_0(\C)$ 
induces a group homomorphism 
\[
f_\X\colon K_0(\C) \to K_0^{\sp}(\X)/\Im(\index_\X \theta_\N^R G_*^{-1}).
\]
And, in \cite[Prop.~3.5]{Fed21}, it is shown that $f_\X$ is an isomorphism if there exists a certain homomorphism in the other direction. 
By \cref{cor:palu-lemma-9-nCT}, we see that 
$f_\X$ is always an isomorphism with inverse induced by $\pi^R_{\N} \index_\X^{-1}$, and hence 
\emph{always} induces an isomorphism 
\[
K_0^{\sp}(\X)/\Im(\index_\X \theta_\N^R G_*^{-1}) \cong K_0(\C).
\]
\end{remark}

Our final goal of this article is to generalize \cite[Thm.~C]{Fed21}, which shows that $K_0(\C)$ is isomorphic to the Grothendieck group arising from $\X$ when it admits an $(n+2)$-angulated structure. However, we need an extra assumption to guarantee the existence of this structure.

\begin{setup}\label{setup:X-closed-under-n-shift}
In addition to \cref{setup:Fedele}, 
we assume $\X[n]=\X$ in the rest of \S\ref{subsec_d_angulated_cat}.
\end{setup}

It follows from \cite[Thm.~1]{GKO13} that $\X$ forms part of an $(n+2)$-angulated category $(\X,[n],\indpent)$.

\begin{construction}\label{const_d_angle}
We recall how to construct $(n+2)$-angles in $\indpent$ from a given morphism $X_1\overset{g_1}{\lra} X_0$ in $\X$.
First we complete $g_1$ to a triangle $C\overset{f_1}{\lra} X_1\overset{g_1}{\lra} X_0\overset{h_1}{\lra} C[1]$ in $\C$.
Since $\X$ is an $n$-cluster tilting subcategory, the object $C$ admits a length $n-1$ $\X$-resolution
\begin{equation}
\label{eq_d_angle}
\begin{tikzcd}[column sep=1.3cm]
    X_{n+1}
        \arrow{r}{f_{n}} 
    & X_n 
        \arrow{r}{f_{n-1}g_n}
    &\cdots 
        \arrow{r}{f_2g_3} 
    & X_2
        \arrow{r}{g_2} 
    &C,
\end{tikzcd}
\end{equation}
with triangles $C_{i+1}\overset{f_i}{\lra} X_i\overset{g_i}{\lra} C_i\overset{h_i}{\lra} C_{i+1}[1]$ for $2\leq i\leq n$ and $(C_2, C_{n+1}) \deff (C, X_{n+1})$.
Here, for each $2\leq i \leq n$,  we have taken right $\X$-approximations 
$X_i\overset{g_i}{\lra}C_i$ 
of $C_i$ with $C_{i+1}\in \X*\cdots *\X[n-i]$.
Since $\X[n]=\X$, we have the $(n+2)$-angle
\begin{equation}\label{eqn:tower-triangles-n-angle}
\begin{tikzcd}[column sep=1.3cm]
    X_{n+1}
        \arrow{r}{f_{n}}
    & X_n 
        \arrow{r}{f_{n-1}g_n}
    & \cdots 
        \arrow{r}{f_2g_3} 
    & X_2
        \arrow{r}{f_1g_2}
    & X_1
        \arrow{r}{g_1}
    & X_0
        \arrow{r}{h}
    & X_{n+1}[n]
\end{tikzcd}
\end{equation}
in $\X$, 
where $h$ is the composition 
$\begin{tikzcd}[column sep=1.5cm, cramped]
    X_0\arrow{r}{h_1} &C_2[1]\arrow{r}{h_2[1]}&C_3[2]\arrow{r}{h_3[2]}&\cdots\arrow{r}{h_n[n-1]} &X_{n+1}[n].
\end{tikzcd}$
\end{construction}

\begin{remark}\label{rem_d_angle}
As the $\X$-resolution \eqref{eq_d_angle} 
is built from 
right $\X$-approximations $g_i$ for $2\leq i \leq n$, 
we have 
$h_i\in[\N]$ for $2\leq i \leq n$. 
In particular, this implies 
$[C]^R_\N = \sum_{i=2}^{n+1}(-1)^i[X_i]^R_\N$
in $K_0(\C^R_\N)$. 
\end{remark}

We denote by $K_0^{\indpent}(\X)$ the Grothendieck group of $(\X,[n], \indpent)$ in the sense of \cite[Def.~2.1]{BT14}.
The following result is a generalization of \cite[Thm.~C]{Fed21}. Indeed, we do not assume that $\C$ is 
$k$-linear ($k$ a field), 
$\Hom$-finite, 
Krull-Schmidt, 
with a Serre functor; nor do we impose any locally bounded assumption on $\X$.

\begin{theorem}\label{thm-general_Fedele-thm}
There is an isomorphism $K_0(\C)\cong K_0^{\indpent}(\X)$ 
induced by \eqref{eqn:triangulated-index}. 
\end{theorem}

\begin{proof} 
Set $\rho \deff \pi^R_{\N} \index_\X^{-1}$. 
It follows from \cref{cor:palu-lemma-9-nCT}
that there is a short exact sequence
\begin{equation}
\begin{tikzcd}
    0
        \arrow{r}
    & \Ker \rho
        \arrow[hookrightarrow]{r}
    & K_0^{\mathsf{sp}}(\X)
        \arrow{r}{\rho}
    & K_0(\C)
        \arrow{r}{}
    & 0.
\end{tikzcd}
\end{equation}
Hence, 
it suffices to show
\[
\Ker\rho=\Braket{ \sum_{i=0}^{n+1}(-1)^i[X_i]^{\sp} | 
X_{n+1}\to\cdots\to X_0\to X_{n+1}[n]\text{\ is an $(n+2)$-angle in\ }\X}.
\]
As each $(n+2)$-angle 
in $\X$ is of the form \eqref{eqn:tower-triangles-n-angle} and, in particular, built from triangles in $\C$, 
we immediately see that 
$\rho\left(\sum_{i=0}^{n+1}(-1)^i[X_i]^{\sp}\right) = 0$. 

For the other containment, we first note that we have
\[
\Ker\rho 
    = \index_\X ( \Ker\pi^R_{\N} )
    = \index_\X ( \Im\theta^R_\N) 
    = \Braket{ \index_\X( [C]^R_\N+[C[-1]]^R_\N ) | C\in\C},
\]
using that $\index_\X$ is an isomorphism, and \cref{prop_comparison} for the second equality. 
Thus, given an object $C_0 \deff C \in\C$, we will construct an $(n+2)$-angle in $\indpent$. 
By taking a right $\X$-approximation 
$g_0\colon X_0 \to C_0$ and a right $\X$-approximation 
$g'_1\colon X_1 \to C_1$, where 
$C_1 \deff \cocone(g_0)$, 
we obtain the following commutative diagram in $\C$ made of triangles.
\[
\begin{tikzcd}
    X_1\arrow[equals]{r}{}\arrow{d}[swap]{g'_1}&X_1\arrow{d}{g_1 \deff f_0 g'_1}&&\\
    C_1\arrow{r}{f_0}\arrow{d}{}&X_0\arrow{r}{g_0}\arrow{d}{}&C_0\arrow{r}{}\arrow[equals]{d}{}&C_1[1]\arrow{d}{}\\
    C_2[1]\arrow{r}{}\arrow{d}[swap]{-f'_1[1]}&C'_0\arrow{r}{s}\arrow{d}{}&C_0\arrow{r}{}&C_2[2]\\
    X_1[1]\arrow[equals]{r}{}&X_1[1]&&
\end{tikzcd}
\]
The second column is an $\mathfrak{s}^R_\N$-triangle  
as $X_1[1]\in\X[1]\sse \X^{\perp_{0}} = \N$ and $\X$ is $n$-cluster tilting. So we have $[C'_0]^R_\N=[X_0]^R_\N - [X_1]^R_\N$.
Following \cref{const_d_angle}, we can complete the morphism 
$\begin{tikzcd}
    X_1\arrow{r}{g_1}& X_0
\end{tikzcd}$
to an $(n+2)$-angle \eqref{eqn:tower-triangles-n-angle}.
Thus, due to \cref{rem_d_angle}, we have
\[
[C'_0]^R_\N+[C'_0[-1]]^R_\N
    = \sum_{i=0}^{n+1}(-1)^i[X_i]^R_\N.
\]
Note that $C_2[1]$ sits in $\X[1]*\cdots *\X[n-2]$ 
and thus $C_2[1], C_2[2]\in \X[1] * \cdots * \X[n-1] = \X^{\perp_{0}} = \N$. 
In particular, the morphism $s\colon C'_0\to C_0$ lies in $\Sn$, and hence $[C'_0]^R_\N+[C'_0[-1]]^R_\N=[C_0]^R_\N+[C_0[-1]]^R_\N$ by \cref{lem_add_formula1}. 
Finally, we see that 
\begin{align*}
\index_\X( [C]^R_\N+[C[-1]]^R_\N )
     &= \index_\X( [C'_0]^R_\N+[C'_0[-1]]^R_\N )\\
     &= \index_\X( \sum_{i=0}^{n+1}(-1)^i[X_i]^R_\N)\\
     &= \sum_{i=0}^{n+1}(-1)^i[X_i]^{\sp}. \qedhere
\end{align*}
\end{proof}

We end this section by making a remark on a class of examples where \cref{thm-general_Fedele-thm} applies but \cite[Thm.~C]{Fed21} may not.

\begin{remark}
As a benefit of our generalization, $n$-cluster tilting subcategories closed under $n$-shifts in singularity categories fall within the scope of \cref{thm-general_Fedele-thm}.
We recall that, for a Noetherian $k$-algebra $\Lambda$ over a field $k$, the singularity category $\mathsf{D}_{sg}(\Lambda)$ is defined to be the Verdier quotient of the bounded derived category $\mathsf{D}^b(\Lambda)$ by the subcategory $\mathsf{perf}(\Lambda)$ of perfect complexes.
Some examples and constructions of $n$-cluster tilting subcategories in $\mathsf{D}_{sg}(\Lambda)$ are given in e.g.\ \cite{Iya18,Kva21}.
It is known that $\mathsf{D}_{sg}(\Lambda)$ is not $\Hom$-finite in general; indeed, the $\Hom$-finiteness of $\mathsf{D}_{sg}(\Lambda)$ implies that $\Lambda$ is Gorenstein in some cases (see \cite[Thm.~2.1]{Kal21}). 
Thus, \cref{thm-general_Fedele-thm} is effective even in the case of singularity categories by passing to the idempotent completion.
\end{remark}

\section{Application to the index in abelian categories}

In this section we will use \cref{thm_Quillen_resolution} to prove an analogue of \cref{cor_JS_index} for abelian categories. 
Since we will not need the details of the definition of an $n$-cluster tilting subcategory of an abelian category, we refer the reader to \cite[Def.~3.14]{Jasso} for the precise formulation. Instead we just recall the needed consequences of the definition below. 
We fix the following setup for the remainder of the article.

\begin{setup}
Suppose $(\C,\BE,\fs)$ is a skeletally small abelian category and that $\X\sse\C$ is an $n$-cluster tilting subcategory ($n \geq 1$). 
\end{setup}

By the dual of \cite[Prop.~3.17]{Jasso}, for each $C\in\C$, 
there is a diagram 
\begin{equation}
\label{eqn:X-res-for-C}
\begin{tikzcd}[column sep=0.5cm]
0
    \arrow{r}
& X_{n-1}
    \arrow{rr}{f_{n-2}}
    \arrow[equals]{dr}
& {}
& X_{n-2}
    \arrow{rr}{f_{n-3}g_{n-2}}
    \arrow{dr}{g_{n-2}}
& {}
& \cdots
    \arrow{rr}{f_{0}g_{1}}
& {}
& X_{0}
    \arrow{rr}
    \arrow{dr}{g_{0}}
& {}
& C 
    \arrow{r}
& 0 
\\
{}
& {}
& C_{n-1}
    \arrow{ur}{f_{n-2}}
& {}
& C_{n-2}
& {}
& C_{1}
    \arrow{ur}{f_{0}}
& {}
& C_{0}
    \arrow[equals]{ur}
& {}
& {}
\end{tikzcd}
\end{equation}
in $\C$, 
where:
\begin{enumerate}
    \item $X_{i}\in\X$ for each $0 \leq i \leq n-1$;

    \item 
    $
    \begin{tikzcd}[column sep=0.5cm]
        0
            \arrow{r}
        & C_{i+1} 
            \arrow{r}{f_{i}}
        & X_{i}
            \arrow{r}{g_{i}}
        & C_{i}
            \arrow{r}
        & 0
    \end{tikzcd}
    $
    is a short exact sequence 
    for each $0 \leq i \leq n-2$; and

    \item the morphism $g_{i}\colon X_{i} \to C_{i}$ is a right $\X$-approximation for each $0 \leq i \leq n-2$.
\end{enumerate}
In particular, we see that each object in $\C$ has a finite $\X$-resolution (of length at most $n-1$) in $(\C,\BE,\fs)$ in the sense of \cref{def_X_resolution}. 
Although $\X$ is also extension-closed and closed under direct summands in $\C$, we cannot apply \cref{thm_Quillen_resolution} yet. Indeed, it easy to find examples where $\X$ is not closed under cocones of $\fs
$-deflations. We will need to pass to a relative structure on $\C$ first.

\begin{definition}(cf.\ \cref{def_relative1})
For objects $A,C\in\C$, we define:
\[
\mathbb{E}_R^\X(C,A) 
    \deff
\Set{\delta = [A\overset{f}{\lra} B\overset{g}{\lra} C]\in\mathbb{E}(C,A) | 
    x^{*}\delta = 0\textnormal{\ for all\ }x\colon X\to C\textnormal{\ with\ }X\in\X}.
\]
\end{definition}

Note that 
$x^{*}\delta = \BE(x,A)(\delta)$ is the pullback of the short exact sequence 
$A\overset{f}{\lra} B\overset{g}{\lra} C$
along $x\colon X\to C$. 
Just like in the triangulated case, $\mathbb{E}_R^\X$ (which is denoted $\BE_{\X}$ in \cite[Def.~3.18]{HLN21}) 
is a closed subfunctor of $\BE$ by \cite[Prop.~3.19]{HLN21}. 
This yields 
the extriangulated category
$\C_R^\X\deff(\C,\mathbb{E}_R^\X,\mathfrak{s}_R^\X)$ which is relative to 
$(\C,\mathbb{E},\mathfrak{s})$. In fact, $\C_R^\X$ is an exact category by \cite[Cor.~3.18]{NP19}, because each $\fs_R^\X$-inflation (resp.\ $\fs_R^\X$-deflation) is a monomorphism (resp.\ an epimorphism).

\begin{lemma}
$\X$ is closed under cocones of $\fs_R^\X$-deflations.
\end{lemma}
\begin{proof}
Suppose 
$\begin{tikzcd}[column sep=0.5cm]
A
    \arrow{r}
& X_{1}
    \arrow{r}
& X_{0}
    \arrow[dashed]{r}{\delta}
& {}
\end{tikzcd}$
is an $\fs_R^\X$-triangle with 
$X_{i}\in\X$. 
Consider the identity morphism $\id_{X_{0}}\colon X_{0} \to X_{0}$ and notice that 
$\delta = \id_{X_{0}}^{*}\delta = 0$ as $X_{0}\in\X$.
In particular, $A$ is a direct summand of $X_{1}$ and hence $A\in\X$.
\end{proof}

Now that we have passed to a relative structure, we must show the $\X$-resolution of $C$ in $(\C,\BE,\fs)$ arising from 
\eqref{eqn:X-res-for-C} is still an $\X$-resolution in $\C_{R}^{\X}$.

\begin{lemma}
Each $\fs$-triangle 
$
\begin{tikzcd}[column sep=0.5cm]
    C_{i+1} 
        \arrow{r}{f_{i}}
    & X_{i}
        \arrow{r}{g_{i}}
    & C_{i}
        \arrow[dashed]{r}{\delta_{i}}
    & {}
\end{tikzcd}
$
arising in \eqref{eqn:X-res-for-C} is in fact an
$\fs_R^\X$-triangle.
\end{lemma}
\begin{proof}
Let $x\colon X \to C_{i}$ be a morphism in $\C$ with $X\in\X$. 
Since $g_{i}$ is a right $\X$-approximation of $C_{i}$, we have that $x$ factors through $g_{i}$. This implies 
$x^{*}\delta_{i} = 0$ by 
\cite[Cor.~3.5]{NP19}, so $\delta_{i}\in\BE_R^\X$. 
\end{proof}

Putting all this together, we immediately have the following by  \cref{thm_Quillen_resolution}.

\begin{theorem}
\label{thm:abelian-n-CT-isomorphism}
There is an isomorphism of abelian groups
\begin{align*}
K_{0}(\C_R^\X) & \overset{\cong}{\longleftrightarrow}K_0^{\mathsf{sp}}(\X)\\
[C]^\X_R & \longmapsto \sum_{i=0}^{n-1}(-1)^i[X_i]^{\mathsf{sp}}  \nonumber\\
[X]^\X_R & \longmapsfrom [X]^{\sp}, \nonumber
\end{align*}
where 
$C\in\C$ admits an $\X$-resolution 
$
X_{n-1}\to \cdots\to X_1\to X_0\to C
$
in $\C^\X_R$ of length $n-1$.
\end{theorem}

We now explain how the above relates to the index defined by Reid \cite{Rei20b}.

\begin{remark}\label{rem:relation-to-Reid-index}
The sum 
$
\sum_{i=0}^{n-1} (-1)^{i} [X_{i}]^{\sp}
$ 
was defined to be the \emph{index $\index_{\X}(C)$ of $C$ with respect to $\X$}; see \cite[Sec.~1]{Rei20b}. 
Since $\C$ is idempotent complete, we can always find a \emph{minimal} $\X$-resolution of $C$ by removing 
trivial summands (i.e.\ 
$\cdots \to 0 \to X \overset{\id_{X}}{\lra} X \to 0 \to \cdots$). 
In particular, the value $\index_{\X}(C)$ in $K_{0}^{\sp}(\X)$ 
does not change when we remove such summands and, since a minimal $\X$-resolution is unique up to isomorphism, 
the index is well-defined; see \cite[Rem.~2.1]{Rei20b}.
\end{remark}


\medskip
\noindent
{\bf Acknowledgements.}
The authors would like to thank Marcel B\"{o}kstedt, Hiroyuki Nakaoka and Katsunori Takahashi for their time, valuable discussions and comments. 
We are also very grateful to the referee for their careful reading of earlier versions of this article, which led to several improvements to the exposition.

Y.\ O.\ is supported by JSPS KAKENHI (grant JP22K13893).
A.\ S.\ is supported by: the Danish National Research Foundation (grant DNRF156); the Independent Research Fund Denmark (grant 1026-00050B); and the Aarhus University Research Foundation (grant AUFF-F-2020-7-16).


\end{document}